\documentclass[12pt]{article}

\usepackage{pdfsync}
\usepackage{amssymb}
\usepackage{eufrak}
\usepackage{amsmath}

\usepackage{color}

\begin{document}

\title{A Jordan-H\"older Theorem for  Differential Algebraic Groups}
\author{Phyllis J. Cassidy\footnote{Department of Mathematics, Smith College, MA 01063; Department of Mathematics, The City College of New York, 
New York, NY 10038, USA, email: \texttt{pcassidy@smith.edu}} \ and 
Michael F. Singer\footnote{Department of Mathematics, North Carolina
State University, Box 8205, Raleigh, NC 27695-8205, USA, email: \texttt{singer@math.ncsu.edu}. The second author was partially supported by NSF Grant CCR-0634123 and he would also like to thank the  Max Planck Institut f\"ur Mathematik  and the Courant Institute of Mathematical Sciences for their hospitality.  This work was begun during a visit to the first institute in the Spring of 2006 and continued during a stay as a Visiting Member at the second institute in the Fall of 2008.} }

\date{}
\def\QED{\hbox{\hskip 1pt \vrule width4pt height 6pt depth 1.5pt \hskip 1pt}}
\newcounter{defcount}[section]
\setlength{\parskip}{1ex}
\newtheorem{thm}{Theorem}[section]
\newtheorem{lem}[thm]{Lemma}
\newtheorem{cor}[thm]{Corollary}
\newtheorem{prop}[thm]{Proposition}
\newtheorem{defin}[thm]{Definition}
\newtheorem{defins}[thm]{Definitions}
\newtheorem{remark}[thm]{Remark}
\newtheorem{remarks}[thm]{Remarks}
\newtheorem{ex}[thm]{Example}
\newtheorem{exs}[thm]{Examples}

\def\GL{{\rm GL}}
\def\SL{{\rm SL}}
\def\Sp{{\rm Sp}}
\def\sl{{\rm sl}}
\def\sp{{\rm sp}}
\def\gl{{\rm gl}}
\def\PSL{{\rm PSL}}
\def\SO{{\rm SO}}
\def\Gal{{\rm Aut}}
\def\Sym{{\rm Sym}}
\def\tildea{\tilde a}
\def\tildeb{\tilde b}
\def\tildeg{\tilde g}
\def\tildee{\tilde e}
\def\tildeA{\tilde A}
\def\calG{{\cal G}}
\def\calH{{\cal H}}
\def\calT{{\cal T}}
\def\calC{{\cal C}}
\def\calP{{\cal P}}
\def\calS{{\cal S}}
\def\calM{{\cal M}}
\def\calF{{\cal F}}
\def\calY{{\cal Y}}
\def\calX{{\cal X}}
\def\calO{{\cal O}}
\def\calR{{\cal R}}
\def\curve{{\rm {\bf C}}}
\def\P1{{\rm {\bf P}}^1}
\def\Ad{{\rm Ad}}
\def\ad{{\rm ad}}
\def\Aut{{\rm Aut}}
\def\Int{{\rm Int}}
\def\CX{{\mathbb C}}
\def\QX{{\mathbb Q}}
\def\NX{{\mathbb N}}
\def\ZX{{\mathbb Z}}
\def\DX{{\mathbb D}}
\def\AX{{\mathbb A}}
\def\UX{{\mathbb U}}
\def\sd{{\sigma\der}}
\def\SDP{{\Sigma\Delta\Pi}}
\def\SD{{\Sigma\Delta}}
\def\der{{\partial}}
\def\dd{{\delta}}
\def\disp{{\rm disp}}
\def\pdisp{{\rm pdisp}}
\def\End{{\rm End}}
\def\kbar{{\overline{k}}}
\def\ubar{{\tilde{u}}}
\def\zbar{{\tilde{z}}}
\def\Gbar{{\overline{G}}}
\def\Ga{{{\mathbb G}}_a}
\def\Gm{{{\mathbb G}}_m}
\def\td{{\tilde{d}}}
\def\frakg{{\mathfrak g}}
\def\frakL{{\mathfrak L}}
\def\frakD{{\mathfrak D}}
\def\frakU{{\UX}}
\def\vy{{\Vec{y}}}
\def\vz{{\Vec{z}}}
\def\va{{\Vec{a}}}
\def\semi{\hbox{${\vrule height 5.2pt depth .3pt}\kern -1.7pt\times $ }}

\newenvironment{prf}[1]{\trivlist
\item[\hskip \labelsep{\bf
#1.\hspace*{.3em}}]}{~\hspace{\fill}~$\square$\endtrivlist}
\newenvironment{proof}{\begin{prf}{Proof}}{\end{prf}}
 \def\square{\QED}
 \newenvironment{proofofthm}{\begin{prf}{Proof of Theorem~\ref{ext}}}{\end{prf}}
 \def\square{\QED}
\newenvironment{sketchproof}{\begin{prf}{Sketch of Proof}}{\end{prf}}

\def\calu{{\cal \frakU}}
\def\si{\sigma}
\def\cee{\CX}
\def\HX{{\mathbb H}}
\def\GX{{\mathbb G}}
\def\Krd{{\rm Kr.dim}}
\def\dcl{{\rm dcl}}
\maketitle

\begin{abstract}{We show that a differential algebraic group can be filtered by a finite subnormal series of  differential algebraic  groups such that successive quotients are almost simple, that is  have no normal subgroups of the same type.  We give a uniqueness result, prove several properties of almost simple groups  and, in the  ordinary differential case, classify almost simple linear differential algebraic groups.}
\end{abstract}
\section{Introduction} 
Let $(k, \dd) $ be a differential field of characteristic zero with derivation $\dd$ and let $k[\dd]$ be the ring of linear differential operators with coefficients in $k$.  It has been known for a long time  (see  \cite{singer_red} for a brief history and \cite{BCW05},\cite{PuSi2003},\cite{Wu05a} for more recent algorithmic results) that one has a form of unique factorization for this ring: 
\begin{quotation} \noindent Given $L \in k[\dd], \ L \notin k$, there exist irreducible $L_i \in k[\dd], i= 1, \ldots , r$ such that $L= L_1 \cdots L_r$. Furthermore, if $L = \tilde{L}_1, \cdots , \tilde{L_s}, \ \tilde{L}_i \in k[\dd], \tilde{L}_i$ irreducible, then $r = s$ and, for some permutation $\sigma$ of the subscripts, there exist nonzero $R_i, S_i \in k[\dd]$ with $\deg(R_i) < \deg(\tilde{L}_i)$ such that $L_{\sigma(i)} R_i = S_i \tilde{L}_i$ for $i = 1, \ldots, r$. \end{quotation}
A direct generalization of this to partial differential operators fails as the following example shows. In \cite{blumberg},  Blumberg considered the following  third order linear partial differential operator (which he attributes to E. Landau)
\begin{eqnarray*} L = \dd_x^3  + x\dd_x^2\dd_y + 2\dd_x^2 + 2(x+1)\dd_x\dd_y +\dd_x +(x+2)\dd_y
\end{eqnarray*}
  over the differential field $\CX(x,y)$ where $\dd_x = \frac{\partial}{\partial x}$ and $\dd_y = \frac{\partial}{\partial y}$. He showed that this operator has two factorizations
\begin{eqnarray*}
L & = & (\dd_x+1)(\dd_x+1)(\dd_x +x\dd_y)  \\ 
 & = & (\dd_x^2 + x\dd_x\dd_y + \dd_x + (x+2)\dd_y)(\dd_x+1) \ .
 \end{eqnarray*}  and that the operator $\dd_x^2 + x\dd_x\dd_y + \dd_x + (x+2)\dd_y$ could not be further factored in any differential extension of $\CX(x,y)$ (this example is also discussed in \cite{grig_schwarz_04}).  

To confront this situation one must recast the factorization problem in other terms.  For ordinary differential operators, the factorization result is best restated in terms of $k[\dd]$-modules.  If one considers the $k[\dd]$-module $M_L=k[\dd]/k[\dd]L$, then the above factorization result can be stated as a Jordan-H\"older type theorem:  $M_L$ has a composition series $M_L=M_0 \supset M_1 \supset \ldots \supset M_r = \{0\}$ where successive quotients $M_{i-1}/M_i$ are simple and any two such series have the same length and, after a possible renumbering, isomorphic quotients. Casting questions in terms of ideals and modules in  the partial differential case yields many interesting results (cf., \cite{grig_schwarz_04,grig-schwarz_07,gri_sch_08,grig_05}, \cite{cl_qua_08}, \cite{pom_01}). Other approaches to factoring linear operators are contained in \cite{beals},\cite{ shem07a,shem07}, \cite{tsarev2000a}. In \cite{tsa09}, Tsarev alludes to an approach to factorization through the theory of certain abelian categories that has points of contact with the approach presented here.

We will take an alternate approach - via differential algebraic groups.  This stems from the seemingly trivial observation that the solutions of a linear differential equation form a group under addition.  In fact, this is an example of a linear differential algebraic group, that is, a group of matrices whose entries lie in a differential field and satisfy some fixed set  of differential equations (here we identify $\{ y \in k \ | \ L(y) = 0 \}$ with the group $\{ \left(\begin{array}{cc} 1 & y\\0&1\end{array}\right) \ | \ L(y) = 0\}$.) We shall  prove a Jordan-H\"older type theorem for differential algebraic groups (of which linear differential algebraic groups are a special case.) This will encompass the factorization result for linear ordinary differential equations, allow us to give a ``factorization'' result for linear partial differential equations and allow us to reduce the study of general differential algebraic groups to the study of {\it almost simple} groups (to be defined below) and their extensions.

The rest of the paper is organized as follows. In the remainder of this section we review some basic definitions and facts from differential algebra. In Section~\ref{section_jht}, we shall review some definitions and facts from differential algebra and the theory of linear differential algebraic groups.  We will state and prove a Jordan-H\"older Theorem for linear differential algebraic groups and show how it applies to the example of Landau above. In Section~\ref{section_asg}, we will discuss the structure of almost simple groups in more detail. 

The authors would like to acknowledge the influence of Ellis Kolchin. The first author discussed her notions of solid and the core (what we call strongly connected and  the strong identity component) with him in the 1980's and he suggested a version of the Jordan-H\"older Theorem below. In addition, we would like to thank William Sit for explaining the connection between Gr\"obner bases of left ideals in the ring of differential operators and characteristic sets of linear differential ideals (see the discussion in Example~\ref{example213}).

\subsection{Differential Algebra Preliminaries}\label{sec1.1}  Throughout, we let $\mathbb{Z}=$ the ring of integers, $\mathbb{Q}=$ the field of rational
numbers, $\mathbb{C}=$ the field of complex numbers.

We refer to \cite{cassidy_quasi, cassidy1,cassidy2,cassidy3,cassidy6, DAAG, kolchin_groups} for the basic concepts of differential
algebraic geometry and differential algebraic groups. All rings have
characteristic zero. We fix a set $\Delta =\left\{ \delta _{1},\ldots
,\delta _{m}\right\} $ of commuting derivation operators. We often use the
prefix $\Delta$- to indicate the $\Delta$-differential structure on a
ring, field, group, \textit{etc}.  If $k$ is a $\Delta$-field, and 
\emph{\ }$R$ and $S$ are $\Delta$-$k$-algebras, a $k$-homomorphism $\varphi
:R\longrightarrow S$ is a $\Delta$-$k$-homomorphism if $\varphi \circ
\delta =\delta \circ \varphi $, $\delta \in \Delta $. Let $k$ be a $\Delta 
$-field. The $k$-algebra of differential operators generated by $\Delta$
is denoted by $k\left[ \Delta \right] $, and the set of monomials $\theta
=\delta _{1}^{e_{1}}\cdots \delta _{m}^{e_{m}}$ in $k[\Delta ]$ by $\Theta $%
. The \emph{order }of $\theta $ is the degree of $\theta $.

If $\eta =\left( \eta _{1},\ldots ,\eta _{n}\right) $ is a family of
elements of a $\Delta$-$k$-algebra, the $\Delta$-$k$-algebra $k\left\{
\eta \right\} $ is the $k$-algebra $k\left[ \Theta \eta \right] $, where $%
\Theta \eta $ is the family of derivatives $\left( \theta \eta _{j}\right)
_{\theta \in \Theta ,1\leq j\leq n}$ of the coordinates of $\eta $. Thus,
if $z_{1},\ldots ,z_{n}$ are elements of a $\Delta$-$k$-algebra such that
the family $\left( \theta z_{j}\right) _{\theta \in \Theta ,1\leq j\leq n}$
is algebraically independent over $k$, the $\Delta$-$k$-algebra $k\left\{
z\right\} =k\left\{ z_{1},\ldots ,z_{n}\right\} $ is called the $\Delta$%
-polynomial algebra$.$ Suppose $k\left\{ \eta \right\} $ is an integral
domain. Set $k\left\langle \eta \right\rangle $ equal to its quotient
field. We say that $k\left\{ \eta \right\} $ (\textit{resp,} $%
k\left\langle \eta \right\rangle $) is a \emph{finitely generated }$\Delta$-%
$k$-algebra (\textit{resp.} $\Delta$-$k$-field ). The quotient field $%
k\left\langle z\right\rangle $ of the $\Delta$-polynomial ring is called
the field of $\Delta$\emph{-rational functions. }Let $\left( P_{i}\right)
_{i\in I}$ be a family of $\Delta$-polynomials in $k\left\{ z_{1},\ldots
,z_{n}\right\} $. The $\Delta$-ideal $\left[ \left( P_{i}\right) _{i\in I}%
\right] $ equals the ideal $\left( \left( \Theta P_{i}\right) _{i\in
I}\right) $.

Fix a $\Delta$-field $k$.  A $\Delta$-$k$-field $K$ is \emph{
differentially closed }if given a prime $\Delta$-ideal $I$ in $K\left\{
z_{1},\ldots ,z_{n}\right\} $ and differential polynomial $Q\notin I$, there
exists $\eta =\left( \eta _{1},\ldots ,\eta _{n}\right) \in K^{n}$, with 
\[
Q\left( \eta \right) \neq 0,\text{ \ \ and }P\left( \eta \right) =0\text{ \
\ for all }P\in I\text{.} 
\]%
Equivalently, $K$ is differentially closed if every system of differential
polynomial equations and inequations that has a solution with coordinates in
a $\Delta$-$K$-field has a solution with coordinates in $K$ (see \cite{kolchin_constrained}, \cite{marker2002}, \cite{mmp}, \cite{pillay97}).  We say that a $\Delta$-$k$-field $\UX$ is {\it $k$-universal} if, in addition to being differentially closed, $\mathbb{U}$ satisfies the useful property that if $I$ is a
prime $\Delta$-ideal in $k\left\{ z_{1},\ldots ,z_{n}\right\} $, there is
in $\mathbb{U}^{n}$ a \emph{generic zero} of $I$. A generic zero $\eta $
is a zero of $I$ \ that has the property that if $P$ is in $k\left\{
z_{1},\ldots ,z_{n}\right\} $ and $P\left( \eta \right) =0$, then $P$ is in $%
I$. We fix $k$ and $\mathbb{U}$.


We define on \emph{affine }$n$-\emph{space }$\mathbb{U}^{n}$ both the
Zariski and Kolchin topologies. The \emph{Kolchin topology}, which is
finer than the Zariski topology, is also Noetherian. A subset $X$ of $%
\mathbb{U}^{n}$ is a $\Delta$-\emph{variety }if it is \emph{Kolchin closed}%
, \textit{i.e.}, $X$ is the set of zeros of a set of $\Delta$-polynomials
with coefficients in $\mathbb{U}$. By the Ritt Basis Theorem, there is a
finite set $P_{1},\ldots ,P_{r}$ of $\Delta$-polynomials with coefficients
in $\mathbb{U}$ such that $X$ is the set of zeros of $P_{1},\ldots ,P_{r}$.
\ Call $P_{1},\ldots ,P_{r}$ \emph{defining }$\Delta$\emph{-polynomials }of%
\emph{\ }$X$. The radical $\Delta$-ideal $I=\surd \left[ P_{1},\ldots
,P_{r}\right] $ in $\mathbb{U}\left\{ y\right\} $ generated by $P_{1},\ldots
,P_{r}$, is called the \emph{defining }$\Delta$-ideal of $X$. If $%
P_{1},\ldots ,P_{r}$ have coefficients in $k$, we say that $X$ is \emph{%
defined over }$k$, and call it a $\Delta$-$k$-variety. Let $Y$ be a $\Delta$-subvariety of a $\Delta$-$k$-variety $X$.  We do not assume that $Y$ is defined over $k$. However, there is a finitely generated delta-field extension $K$ of $k$ such that $Y$ is a $\Delta$-$K$-variety. $X$ is \emph{%
irreducible} if and only if $I$ is prime. The Ritt Basis Theorem implies
that $X$ is a finite union of maximal irreducible $\Delta$-subvarieties.  This implies that the Kolchin topology is Noetherian.
In this paper, we will usually assume that $X$ is irreducible, with defining 
$\Delta$-ideal $I$. The {\it residue class ring} $\mathbb{U}\left\{ z_1,\ldots
,z_{n}\right\} /I$ is a $\Delta$-$\mathbb{U}$-algebra and an integral
domain, which we denote by $\mathbb{U}\left\{ X\right\} $. As in algebraic
geometry, we regard the residue classes of the $\Delta$-indeterminates as
coordinates on $X$, and call the elements of $\mathbb{U}\left\{ X\right\} $ $%
\Delta$\emph{-polynomial functions} on $X$.  If $X$ is a $\Delta$-$k$-variety, the $\Delta$-$k$-subalgebra $k\left\{ X\right\} $, consisting of
the elements of $\mathbb{U}\left\{ X\right\} $ that have coefficients in $k$%
, is $\Delta$-$k$-isomorphic to $k\left\{ z_{1},\ldots ,z_{n}\right\}
/(I\cap k\left\{ z_{1},\ldots ,z_{n}\right\} ).$ If $X$ is irreducible, the elements of the
quotient field $\mathbb{U}\left\langle X\right\rangle $ of $\mathbb{U}%
\left\{ X\right\} $ (\textit{resp. }$k\left\langle X\right\rangle $ of $%
k\left\{ X\right\} $) are called $\Delta$-rational functions (\textit{resp.~}$\Delta$-$k$-rational functions) on $X$.  A $\Delta$-rational function $%
f $ in $\mathbb{U}\left\langle X\right\rangle $ is \emph{everywhere defined}
if for every $\eta \in $ $X$, there exist $p,q\in \mathbb{U}\left\{
X\right\} $ such that $q\left( \eta \right) \neq 0$, and $f\left( \eta
\right) =\frac{p\left( \eta \right) }{q\left( \eta \right) }$. In contrast
with rational functions on algebraic varieties, an everywhere defined $%
\Delta $-rational function need not be in the ring of $\Delta$-polynomial
functions on $X$.  Indeed, it may not have a global denominator.  Let $%
\eta =\left( \eta _{1},\ldots ,\eta _{n}\right) $ be in $\mathbb{U}^{n}$,
and let $\sigma $ be a $\Delta$-$k$-automorphism of $\mathbb{U}$.  We
define $\sigma \eta =\left( \sigma \eta _{1},\ldots ,\sigma \eta _{n}\right) 
$. $\sigma \left( X\right) $ is a $\Delta$-variety.  A necessary and
sufficient condition that $X$ be defined over $k$ is that $\sigma \left(
X\right) \subseteq X$ (equivalently, $\sigma \left( X\right) =X$) for every $%
\Delta $-$k$-automorphism of $\mathbb{U}$. 

If $X$ is an irreducible $\Delta$-$k$-variety, an element $\eta =\left(
\eta _{1},\ldots ,\eta _{n}\right) $ is generic over $k$ for $X$ if $\eta $
is a generic zero of $I\cap k\left\{ z_{1},\ldots z_{n}\right\} $, where $I$
is the defining $\Delta$-ideal of $X$. The $\Delta$-field $k\left\langle
\eta \right\rangle $ is $\Delta$-$k$-isomorphic to $k\left\langle
X\right\rangle $.  It represents in $\mathbb{U}$ the field of rational
functions on $X.$

If $X\subseteq \mathbb{U}^{n}$ and $Y\subseteq \mathbb{U}^{p}$ are
irreducible $\Delta$-$k$-varieties, a $\Delta$\emph{-morphism} $%
f:X\longrightarrow Y$ is a map whose $p$ coordinate functions are everywhere
defined $\Delta$-rational functions. If they are $\Delta$-$k$-rational,
we call $f$ a $\Delta$-$k$-morphism. We should note that if $f$ is a $%
\Delta $-morphism, $f\left( X\right) $ and its Kolchin closure are
irreducible, and are defined over $k$ if $f$ is a $\Delta$-$k$-morphism.
A $\Delta$\emph{- morphism }$f:X\longrightarrow Y$ is \emph{dominant }if
the Kolchin closure of $f\left( X\right) $ is $Y$. In this case, $f\left(
X\right) $ contains a Kolchin dense open subset of $Y$ ($f\left( X\right) $
is \emph{constructible}). If $X$ is an arbitrary $\Delta$-$k$-variety
with components $X_{1},\ldots ,X_{s}$, and $Y$ is a $\Delta$-$k$-variety
with components $Y_{1},\ldots ,Y_{t}, $ we can define a $\Delta$-morphism to be an $s$-tuple of $\Delta$-morphisms $%
f_{i}:X_{i}\longrightarrow Y_{j_{i}\text{ }}$, $i=1,\ldots ,s$. 

If $X\subseteq \mathbb{U}^{n}$and $Y\subset \mathbb{U}^{p}$ are $\Delta$%
-varieties, then the Cartesion product $X\times Y$ is a $\Delta$-subvariety
of $\mathbb{U}^{n+p}$, and the projection maps are $\Delta$-morphisms.

An \emph{affine }$\Delta$-group is a group $G$ whose underlying set is a $%
\Delta $-subvariety of $\mathbb{U}^{n}$ for some $n$, and whose group laws
are $\Delta$-morphisms.  A $\Delta$-homomorphism of affine $\Delta$%
-groups is a homomorphism that is also a $\Delta$-morphism of $\Delta$%
-varieties. As in algebraic group theory, $G$ has a finite number of
connected components. We call $G$ a $\Delta$-$k$-group if its group laws
and components are defined over $k$ and its identity element has coordinates
in $k$. Note that a $\Delta$-subgroup of $G$ is not necessarily a $\Delta$-$k$-subgroup, but is defined over a finitely generated $\Delta$-$k$-field in $\UX$. Not all affine $\Delta$-groups are linear (that is, isomorphic to a $\Delta$-subgroup of $\GL_n(\UX)$ for some $n$). In this paper, we
will assume that a linear $\Delta$-group is embedded in $\GL_{n}( 
\mathbb{U})$, for some $n$.

In \cite{kolchin_groups}, Kolchin develops axiomatically a theory of $\Delta$-$k$-groups and 
$\Delta$-$k$-varieties, extending the affine theory. In Chapter V.3 of
\cite{kolchin_groups}, Kolchin proves that a $\Delta$-$k$-variety $X$ has a $k$-affine open
dense subset $U$. In the corollary, p. 140, he shows that by extending $k$
perhaps to a finitely generated $\Delta$-extension field, we may assume
that there is a covering of $X$ by open dense subsets, each of which is
defined over $k$ and $k$-affine. Although we will state and prove our
results in the general context of \cite{kolchin_groups}, readers who wish to restrict
themselves to \emph{linear }$\Delta$-groups, or just to the solution sets
of linear homogeneous differential equations, will find almost all that is
needed (with one notable exception which we will discuss later) in the
papers \cite{cassidy1,cassidy2}. Another source is Anand Pillay's foundational paper
\cite{pillay97}, which treats $\Delta$-groups as \emph{definable groups }in
differentially closed fields$.$The theory of definable groups is
particularly well-suited to interpret Kolchin's axiomatic approach.

  \section{Jordan-H\"older Theorem}\label{section_jht} 
\subsection{Differential Type and Typical Differential Transcendence Degree}\label{subsec2.1}
The key concepts used in this paper are \emph{differential type }and \emph{%
typical differential transcendence degree} (typical differential dimension).
 We briefly review their definitions.  Let $k$ be a $\Delta $-field
finitely generated over $\mathbb{Q}$. \\[0.1in]
Let $\eta =\left( \eta _{1},\ldots ,\eta _{n}\right) $ be in $%
\mathbb{U}^{n}$ and let $\Theta \left( s\right) $ be the set of monomial
operators in $\Theta $ of order $\leq s$. In Section II.12 of \cite{DAAG},
Kolchin shows that there is a numerical polynomial $\omega _{\eta/k}(s)$ (%
\textit{i.e.}, a polynomial taking on integer values on $\mathbb{Z}$),
called the \emph{differential dimension polynomial}, of degree $\leq m$,
such that for large values of $s$, $\omega _{\eta/k}(s)$ is the
transcendence degree over $k$ of $k\left( \left( \theta \eta _{i}\right)
\right) _{\theta \in \Theta \left( s\right) ,1\leq i\leq n}$ for all
sufficiently large values of $s$.  See Example~\ref{example213} for further discussion
about the computation of the differential dimension polynomial, and also
 \cite{DAAG} for further information, and algorithms, as well as references.  We
may write%
\[
\omega _{\eta/k}(s)=\sum_{i=0}^{m}a_{i}\binom{s+i}{i}, \ \ a_{i}\in \mathbb{Z}. 
\]%
If $I$ is a prime $\Delta $-ideal in the differential polynomial ring $%
\mathbb{U}\left\{ z_{1},\ldots ,z_{n}\right\} $, and $k$ is a differential
field of definition of $I$, we may define the differential dimension
polynomial of $I$ to be $\omega _{\eta/k}(s),$where $\eta $ is a generic
zero of $I\cap k\left\{ z_{1},\ldots ,z_{n}\right\} $. $\omega _{\eta/k}(s)$
is independent of the choice of $k$ and $\eta $.  If $X$ is the Kolchin
closed subset of $\mathbb{U}^{n}$ defined by $I$, and $\eta $ is generic for 
$X$ over a field $k$ of definition, we define the dimension polynomial $%
\omega \left( X\right) $ to be $\omega _{\eta/k}(s)$.\\[0.1in] 
The dimension polynomial $\omega \left( X\right) $ is not, however, a
differential birational invariant. This means that it makes no sense to
speak of the dimension polynomial of an arbitrary irreducible $\Delta $-$k$-variety. There are, however, two important differential birational
invariants singled out by Kolchin:  the \emph{degree }$\tau $ and the\emph{%
\ leading coefficient }$a_{\tau }$ of $\omega \left( X\right) $.  If $X$ is
any irreducible $\Delta $-$k$-subvariety of $\mathbb{U}^{n}$, we call the
degree of $\omega \left( X\right) $ the \emph{differential type }(or $\Delta 
$-type) $\tau \left( X\right) $ \emph{of }$X$, and the leading coefficient
the \emph{typical differential dimension }(or $\Delta $-dimension) $a_{\tau
}\left( X\right) $ \emph{of }$X$. Let $X$ be a $\Delta $-$k$-variety. As
we mentioned in Section~\ref{sec1.1}, we may assume that $k$ is such that there is a
finite set of Kolchin open $k$-affine subsets $U$ of $X$ covering $X$. 
There exists a positive integer $n$ such that each open set $U$ in the
covering is $\Delta $-$k$-rationally isomorphic to a $\Delta $-$k$%
-subvariety $Y$ of $\mathbb{U}^{n}$.  If $X$ is irreducible, so are $U$ and 
$Y$.  In this case, we define $\tau \left( U\right) $ to be $\tau \left(
Y\right) $, and $a_{\tau }\left( U\right) $ to be $a_{\tau }\left( Y\right) $%
. Since $\Delta $- type and typical $\Delta $-dimension are differential
birational invariants, all the open sets in the covering have the same $%
\Delta $- type $\tau $ and typical $\Delta $-dimension $a_{\tau }$.  So,
we can define $\tau \left( X\right) =\tau $ and $a_{\tau }\left( X\right)
=a_{\tau }$. We should note that a point of $X$ is generic for $X$ over $k$
if and only if whenever $U$ is a $k$-affine $k$-open subset of $X$
containing the point, and $Y$ is a $\Delta $-$k$-subvariety of $\mathbb{U}%
^{n}$ that is $\Delta $-$k$-rationally isomorphic to $U$, the coordinate $n$-tuple $\eta =\left( \eta _{1},\ldots ,\eta _{n}\right) $  is
generic over $k$ for $Y$. By abuse of language, we will call the generic
point $\eta $.  Thus, $\tau \left( X\right) $ is the degree and $a_{\tau
}\left( X\right) $ is the leading coefficient of $\omega_{\eta/k}\left(
s\right) $. If $X$ is not irreducible, we define $\tau= \tau \left( X\right) $ to
be the maximum of the differential type of its irreducible components, and $%
a_{\tau }\left( X\right) $ to be the maximum of their typical differential
dimensions.  The connected components of a $\Delta $-$k$-group $G$ all have
the same $\Delta $-type and typical $\Delta $-dimension.  We should also
mention that if $X$ is an irreducible $\Delta $-subvariety of a $\Delta $-$k$-group or $\Delta $-$k$-homogeneous space and $\sigma $ is a $\Delta $-automorphism of $\mathbb{U}$ over $k$, then $\tau \left( \sigma \left(
X\right) \right) =\tau \left( X\right) $ and $a_{\tau }\left( \sigma \left(
X\right) \right) =a_{\tau }\left( X\right) $.\\[0.1in]
Grigoriev and Schwarz, \cite{grig_05,grig_schwarz_04, grig-schwarz_07, gri_sch_08}, have made extensive use of
differential type $\tau $ and typical differential dimension $a_{\tau }$ in
their studies of systems of linear homogeneous partial differential
equations. \ Following Grigoriev and Schwarz, we will call the pair $\left(
\tau \left( X\right) ,a_{\tau }\left( X\right) \right) $ the \emph{gauge }\
of $X$.\\[0.1in]
Let $C_{k}$ be the field of $\Delta $-constants of $k$.  A\emph{\
transformation} of $\Delta $ is a set of derivation operators on $k$,
denoted by the symbol $\Delta ^{c}$, where $c=\left( c_{ij}\right) _{1\leq
i,j\leq m}\in GL_{m}\left( C_{k}\right) $ and $\Delta ^{c}$ is the set 
\[
\delta _{i}^{^{\prime }}=\sum_{j=1}^{m}c_{ij}\delta _{j},\ \ 
i=1,\ldots ,m
\]%
Every $\Delta $-field extension of $k$ is also a $\Delta ^{c}$-extension of $%
k$. \ Suppose $L=k\left\langle \eta \right\rangle $, \ with $\eta =\left(
\eta _{1},\ldots ,\eta _{n}\right) $ a finite family of elements of $L$. 
In Theorem 7 of \cite{DAAG}, Kolchin established the significance of $\Delta $-type
and typical $\Delta $-transcendence degree by showing that if $\tau =\tau
_{\eta/k}$ is the $\Delta $-type of $\eta/k$, then there is a
transformation $\Delta ^{c}$ of $\Delta $ and a subset $\Delta ^{\prime }$
of cardinality $\tau $ such that $L$ is a finitely generated $\Delta
^{\prime }$-field extension of $k$ of $\Delta ^{\prime }$-differential
transcendence degree $a_{\tau }$. \ In fact, there is a Zariski open dense
subset of $GL_{m}\left( C_{k}\right) $ with this property.  Since every $%
\Delta $-extension field is a $\Delta ^{c}$-field extension and conversely,
and $\Delta $-dimension polynomials of $\Delta $-$k$-varieties are invariant
under such transformations, we may assume that $\Delta ^{c}=\Delta $, and,
thus, $\Delta ^{\prime }\subset \Delta $. \ Let $K$ be a $\Delta $-subfield
of $L$ containing $k$. \ Then, $K$ is a $\Delta ^{\prime }$-finitely
generated $\Delta ^{\prime }$-extension field of $k$ (Proposition 14, p.
112 of \cite{DAAG}).\\[0.1in]
Since the $\Delta $-type of $\eta/k$ and the typical $\Delta $%
-transcendence degree of $\eta/k)$ are differential birational invariants,
we can refer to them as the $\Delta $-type of $L$ over $k$ and the typical $%
\Delta $-transcendence degree of $L$ over $k.$\\[0.1in]
From the classical interpretation of $\Delta $-type and typical $\Delta $%
-dimension, one expects the truth of the following lemma.

\begin{lem}
Let $L$ be a $\Delta $-finitely generated $\Delta $-extension field of $k$,
and let $K$ be a $\Delta $-subfield of $L$ containing $k$. Then, the $%
\Delta $-type of $K/k$ is less than or equal to the $\Delta $-type of $L/k$,
and the typical $\Delta $-transcendence degree of $K/k$ is less than or
equal to the typical $\Delta $-transcendence degree of $L/k$.
\end{lem}

\begin{proof}
Let $\Delta ^{\prime }\subset \Delta $ be such that $L/k$ is a $\Delta
^{\prime }$-finitely generated extension field. Then, since $K/k$ is also $\Delta ^{\prime }$%
-finitely generated, if $\Delta ^{\prime \prime }$ is any subset of $\Delta $
of larger cardinality than $\Delta ^{\prime }$, $K$ is $\Delta ^{\prime
\prime }$-algebraic over $k$ (\cite{DAAG}, Proposition 12, p. 109. See also
\cite{sit74}, Proposition 2.4, p. 478.) \ It follows that the $\Delta $-type of $K/k 
$ is less than or equal to the $\Delta $-type of $L/k$. If the inequality
of $\Delta $-types is strict, so is the inequality of $\Delta $%
-transcendence degrees. So, suppose $\Delta $-type $($ $K/k)=\Delta $-type 
$($ $L/k)$. Then, the typical $\Delta $-transcendence degree of $K/k$ is
the cardinality of a $\Delta ^{\prime }$-transcendence basis of $K/k$. \
Since it can be extended to a $\Delta ^{\prime }$-transcendence basis of $%
L/k $ (Theorem 4(d), p. 105 of \cite{DAAG}), it follows that the $\Delta $%
-transcendence degree of $K/k$ must be less than or equal to that of $L/k$.
\end{proof}

\begin{cor}
Let $X$ and $Y$ be $\Delta $-$k$-varieties, and let $f:X\longrightarrow Y$
be a dominant $\Delta $-$k$-morphism. Then, $\tau \left( X\right) \geq
\tau \left( Y\right) $ and $a_{\tau }\left( X\right) \geq a_{\tau }\left(
Y\right) $.
\end{cor}

\begin{proof}
Let $\eta$ be generic for $X$ over $k$. Then, $\zeta=f\left( \eta\right) $ is
generic for $Y$ over $k$. Therefore, $k\left\langle \zeta\right\rangle $ is a $%
\Delta $-subfield of $k\left\langle \eta\right\rangle $. The conclusion then
follows from the preceding lemma.
\end{proof}

\begin{lem}Let $X$ and $Y$ be irreducible $\Delta$-$k$-varieties and $f:X\rightarrow Y$ be an injective dominant  $\Delta$-$k$-morphism.  Then $\tau(X) = \tau(Y)$.
\end{lem}
\begin{proof} Let $\eta=(\eta_1, \ldots \eta_n)$ be a generic point of $X$ over $k$. By assumption $\zeta = f(\eta)$ is a generic point of $Y$ over $k$ and, since $f$ is differentially rational, $k<\zeta> \subset k<\eta>$. Note that $\frakU$ is universal over $k<\zeta>$ as well. We will show that $k<\zeta> = k<\eta>$. Let $\sigma$ be a differential isomorphism of $k<\eta>$ over $k<\zeta>$ into $\frakU$. Since $\sigma(\zeta) = \zeta$ and  $\sigma(\zeta) =\sigma(f(\eta)) = f(\sigma(\eta))$ we have that  $f(\sigma(\eta)) = f(\eta)$.  Since $f$ is injective, we have that $\sigma(\eta) = \eta$. If some coordinate $\eta_i$ of $\eta$ were not in $k<\zeta>$, there would be an isomorphism of $k<\eta>$ over $k<\zeta>$ into $\frakU$ which would move $\eta_i$.  Therefore we can conclude that $k<\eta>=k<\zeta>$. Proposition 15(b), p. 117 of \cite{DAAG} implies the conclusion.\end{proof}
\begin{cor}\label{taucor} Let $X$ and $Y$ be irreducible $\Delta$-$k$-varieties  and $f:X\rightarrow Y$ be an injective  $\Delta$-$k$-morphism. Then $\tau(X) \leq \tau(Y)$.
\end{cor}
\begin{proof} Let $V$ be the closure of $f(Y)$.  The map $f:X \rightarrow V$ is injective and dominant so $\tau(X) = \tau(V)$. Proposition 15(b), p. 117 of \cite{DAAG} implies that $\tau(V) \leq \tau(Y)$.\end{proof}
Let $G$ be a $\Delta$-$k$-group and $H$ a $\Delta$-$k$-subgroup of $G$.  Kolchin showed (\cite{kolchin_groups} Ch. IV.4) that the coset space $G/H$ has the structure of a $\Delta$-$k$-variety (in fact, a $\Delta$-$k$-homogeneous space for $G$) and  showed that
\begin{eqnarray}\label{eqn1}
\tau(G) & = & \max(\tau(H), \tau(G/H)).
\end{eqnarray}
Furthermore, he showed that if $\tau = \tau(G)$, then
\begin{eqnarray}\label{eqn2}
a_\tau(G) & = & a_\tau(H) + a_\tau(G/H)
\end{eqnarray}
where $a_\tau(H) = 0$ if $\tau(G)> \tau(H)$  with a similar convention for $a_\tau(G/H)$.

\begin{remarks}\label{rem1}{\rm 1. The fact that $G/H$ has the structure of a $\Delta$-$k$-variety also follows from the model-theoretic concept of {\it elimination of imaginaries} which holds for differentially closed fields (\cite{mmp}, p. 57).\\[0.1in]
2. The simplest (to define, not necessarily to understand) examples of differential algebraic groups are zero sets of systems of linear homogeneous differential equations in one indeterminate. Cassidy \cite{cassidy1} has shown that these are the only $\Delta$-subgroups of $\Ga(\frakU) = (\frakU,+)$  In this case, quotients are easy to construct. If $H\subset G\subset \Ga(\frakU)$ and $H$ is defined by the vanishing of $L_1, \ldots, L_t$, then the image in $\frakU^t$ of $G$ under the map $y\mapsto (L_1(y), \ldots , L_t(y))$ is a $\Delta$-subgroup of $\frakU^n$ isomorphic to $G/H$.   \\[0.1in]
3. As we mentioned earlier, if one restricts one's attention to {\it linear} differential algebraic groups (as defined by Cassidy \cite{cassidy1}), one will find all needed facts in  the works of Cassidy and \cite{DAAG} except for the equations~(\ref{eqn1}) and (\ref{eqn2}), which one can just assume when reading the following results and proofs. }
\end{remarks}
\subsection{Strongly Connected Groups, Almost Simple Groups and Isogeny}
\subsubsection{Strongly Connected Groups}We now define the {\it strong identity component} of a $\Delta$-$k$-group $G$. Let $\calS$ be the set of all $\Delta$-subgroups $H \subset G$ such that $\tau(H) = \tau(G) $ and $a_\tau(H) =a_\tau(G)$. Note that from (\ref{eqn1}) and (\ref{eqn2}), this is the same as the set of $\Delta$-subgroups $H\subset G$ such that $\tau(G/H) < \tau(G)$.   We claim that if $H_1, H_2 \in \calS$, then $H_1 \cap H_2 \in \calS$.  To see this, note that we have an injection of $H_1/H_1\cap H_2$ into $G/H_2$, so by (\ref{eqn1}) we have $\tau(H_1/H_1\cap H_2) < \tau(G)$. Again by (\ref{eqn1}), we have that $\tau(G/H_1\cap H_2) = \max(\tau(G/H_1), \tau(H_1/H_1\cap H_2)) < \tau(G)$. Therefore $H_1 \cap H_2 \in \calS$.  \\[0.1in]
Since the Kolchin topology is Noetherian, the set $\calS$ has minimal elements.  The above argument shows that there is a unique minimal element  and justifies the following definition.
\begin{defin} Let $G$ be a $\Delta$-$k$-group. The {\rm strong identity component} $G_0$ of $G$ is defined to be the smallest $\Delta$-$\frakU$-subgroup $H$  of $G$ such that $\tau(G/H) < \tau(G)$.  We say that $G$ is {\rm strongly connected} if $G_0 = G$.
\end{defin}
Note that $G_0$ is contained in the identity component $G^0$.  Therefore, every strongly connected $\Delta$-group is connected.
\begin{remarks}\label{remarksolid}{\rm 1.~In the usual theory of algebraic groups, one can define the identity component of a group $G$ to be the smallest subgroup $G^0$ such that $G/G^0$ is finite and define a group to be connected if $G = G^0$. The notions of strong identity component and strongly connected are meant to be refinements in the context of $\Delta$-groups of these former concepts.\\[0.1in]
2. If $\sigma$ is a $\Delta$-$k$-automorphism of $\frakU$, then $\sigma(G_0)$  is again a minimal element of $\calS$ and so must coincide with $G_0$.  Therefore, Corollary 2, p.77 of \cite{kolchin_groups} implies that $G_0$ is defined over $k$.  Note that if $\phi:G\rightarrow G$ is a $\Delta$-automorphism of $G$, then $\phi(G_0)$ is again a minimal element of $\calS$ and so must equal $G_0$. Therefore $G_0$ is a characteristic subgroup of $G$. In particular  $G_0$ is a {\it normal} subgroup of $G$ and if $H$ is a normal $\Delta$-subgroup of $G$, the strong identity component of $H$ is again normal in $G$.\\[0.1in]
3.  The analogue of the fact that if $G$ is an algebraic group, every
connected algebraic subgroup is contained in the identity component of $G$
is the following: \ If $G$ is a $\Delta$-group of type $\tau $, then every
strongly connected $\Delta$-subgroup $H$ of type $\tau $ is contained in the strong identity component of $G$.  To see this, first note that  the usual isomorphism theorems hold for $\Delta$-groups (\cite{kolchin_groups}, Chapter IV). Since $G_0$ is normal in $G$, $HG_0$ is a $\Delta$-subgroup of $G$ and $HG_0/G_0$ is isomorphic to $H/G\cap G_0$ as $\Delta$-groups. Since the canonical inclusion  $H/H\cap G_0 \hookrightarrow G/G_0$ is injective, Corollary~\ref{taucor} implies that $\tau(H/H\cap G_0) < \tau(G) = \tau(H)$.  Since $H$ is strongly connected, $H\cap G_0 = H$.    \\[0.1in]
4. The analogue of the fact that if $G$, $G^{\prime }$ are algebraic
groups, and $\phi :G\longrightarrow G^{\prime }$ is a homomorphism, then,
the image of the identity component is contained in the identity component
of $G^{\prime }$ is the following: \ If $G$, $G^{\prime }$ are $\Delta$%
-groups of the same type, and $\phi :G\longrightarrow G^{\prime }$ is a
homomorphism, then, the image of the strong identity component of $G$ is contained in the strong identity component of 
$G^{\prime }$. This follows immediately from the preceding remark.\\[0.1in]
5. Let $G$ be a strongly connected $\Delta$-group.  Let $G^{\prime }$ be a nontrivial $\Delta $-group, and let $\varphi
:G\longrightarrow G^{\prime }$ be a surjective $\Delta $-homomorphism. 
Then, $G^{\prime }$ is strongly connected and the type $\tau \left(
G^{\prime }\right) $ equals the type $\tau $ of $G$.  For, Corollary 2.2
says that $\tau \left( G^{\prime }\right) \leq \tau \left( G\right) =\tau $.
 Therefore, the strong connectivity of $G$ implies that $\tau \left(
G^{\prime }\right) =\tau $.  Suppose $G^{\prime }$ is not strongly
connected. Then, there exists a nontrivial $\Delta $-group $G^{\prime
\prime }$ such that $\tau \left( G^{\prime \prime }\right) $ is less than $%
\tau $, and a surjective $\Delta $-homomorphism $\varphi ^{\prime
}:G^{\prime }\longrightarrow G^{\prime \prime }$.  Then, $\varphi ^{\prime
}\circ \varphi $ is a surjective $\Delta $-homomorphism from $G$ onto $%
G^{\prime \prime }$, contradicting the strong connectivity of $G$. }
\end{remarks}
The following lemma is useful in showing a group is strongly connected.
\begin{lem} \label{stronglem} Let $G$ be a $\Delta$-group and let $H$ and $N$ be strongly connected subgroups, with $N\lhd G$ and $\tau(H) = \tau(N)$. Then, the $\Delta$-subgroup $HN$ of $G$ is strongly connected and $\tau(HN) = \tau(H) = \tau(N)$. Furthermore, $a_\tau(HN) = a_\tau(H) + a_\tau(N) -a_\tau(H\cap N)$. \end{lem}
\begin{proof} Without loss of generality, we may assume that  $H \not\subset N$. We have $\tau(HN) = \max \{\tau(HN/N),\tau( N)\} = \max\{\tau(H/H\cap N),\tau(N)\}$. Since $H$ is strongly connected and $H\cap N \neq H$, we have $\tau(H/H\cap N) = \tau(H)$. Therefore $\tau(HN) = \tau(H) = \tau(N)$. Remark~\ref{remarksolid}.3 implies that $H$ and $N$ are contained in $(HN)_0$. This implies that $HN \subset (HN)_0$ so $HN$ is strongly connected. The final statement concerning the typical dimension (as well as the statement concerning differential type) is contained in (\cite{sit74} Cor.~4.3, p.~485; see also \cite{kolchin_groups}, Cor.~4, p.~109). \end{proof}  
\begin{exs}\label{example1} {\rm  Let $\UX$ be an  ordinary universal differential field with derivation $\dd$. \\[0.1in]
1.  Let $G$ be the additive group $\Ga(\UX)$.  $G$ has type $1$ and typical differential dimension $1$.  Cassidy \cite{cassidy1} showed that any proper $\Delta$-subgroup $H$ is of the form $H = \{a \in k \ | \ L(a) = 0\}$ for some linear differential operator $L$. The type of such a group is $0$ and its typical differential dimension is $d$ where $d $ is the order of $L$.  Therefore the type of $G/H$ for any proper $H$ is equal to the type of $G$ and so $G$ is strongly connected.\\[0.1in]
2. Let $G_1 = G$ as above and $G_2 = \Ga(C)$ where $C = \{c \in k \ | \dd(c)=0\}$.  The type of $G_2$ is $0$ and the typical differential dimension  is $1$.  Using equations (\ref{eqn1}) and (\ref{eqn2}), we see that $\tau(G_1\times G_2) = \max(\tau({G_1}),\tau({G_2})) = 1$ and $a_\tau(G_1\times G_2) = a_\tau(G_1) = 1$. The group $G_1$ is minimal with respect to having the same type and typical differential dimension as $G_1\times G_2$ and so is the strong identity component of this group.  Note that $G_1\times G_2$ is connected.
}
\end{exs}
To state the analogue of the Jordan-H\"older Theorem, we need two  more definitions: the first is the appropriate notion of ``simple'' and the second is the relevant notion of ``equivalent''. The next two subsections deal with these notions.
\subsubsection{Almost Simple Groups}
\begin{defin} An infinite $\Delta$-group $G$ is  {\rm almost simple} if for any normal proper $\Delta$-subgroup $H$ of $G$ we have $\tau(H) < \tau(G)$. A $\Delta$-$k$-group is {\rm almost $k$-simple} if for any normal $\Delta$-$k$-subgroup $H$ of $G$ we have $\tau (H) < \tau (G)$.
\end{defin}
If $G$ is almost simple, then $G$ is almost $k$-simple. Note that an almost simple group is strongly connected. We shall use the term {\it simple}  {\ to denote a group (resp.,  algebraic group, $\Delta$-group) whose only proper normal subgroup (resp., normal algebraic subgroup, normal $\Delta$-subgroup) is the trivial group, and the term {\it quasisimple} to denote a group (resp., algebraic group, $\Delta$-group) $G$ such that $G/Z(G)$ is simple and $Z(G)$ is finite, where $Z(G)$ is the center of $G$. If $G$ is an algebraic $k$-group (resp.  $\Delta$-$k$-group), we will call $G$ $k$-simple ($k$-quasismple) if every proper normal algebraic $k$-subgroup (resp. $\Delta$-$k$-subgroup) is trivial (resp. finite).
\begin{exs}\label{example29}{\rm Let $\UX$ be as in Example~\ref{example1} and $C_\UX$ its field of constants.\\[0.1in]
1. The group $\Ga(C_\UX)$ has type $0$ with only the trivial group as a proper subgroup so it is  simple (and so almost simple).\\[0.1in]
2. The group $\Gm(C_\UX)$ has type $0$ with only finite subgroups as proper algebraic (or $\delta$-) subgroups so it is quasisimple.\\[0.1in]
3. The group $\Ga(\UX)$ has type $1$ and any proper subgroup has type $0$ or $-1$ (if it is $\{0\}$). The proper $\delta $-subgroups of $G=\mathbb{G}_{\text{a }}\left( \mathbb{U}%
\right) $ are the kernels of nonzero linear differential operators $L\in $ $%
\mathbb{U}\left[ \delta \right] $. \ It follows, as E. Cartan noted in \cite{cartan09},
that every nontrivial $\delta $-homomorphic image of $G$ is isomorphic to $G
$. Therefore $\Ga(\UX)$ is almost simple.\\[0.1in]
4. If $G$ is a quasisimple linear algebraic group defined over $C_\UX$, then, as we shall show in Section~\ref{section_asg} that $G(\UX)$ and $G(C_\UX)$ are quasisimple $\delta$-groups. \\[0.1in]
In Section~\ref{section_asg}, we shall show that for ordinary differential fields, these are  essentially the only almost simple {\it linear} $\delta$-groups.}
\end{exs}
\begin{ex}\label{evolex}{ \rm Let $\Delta = \{\dd_t,\dd_x\}$,$\UX$ be a universal $\Delta$-field and $z$ a differential indeterminate. Let $G \subset \Ga(\UX)$ be the additive group of solutions of $\dd_ty = f(y), \ f\in \UX[\dd_x]$.  Following Proposition 2.45 of \cite{suer07}, we show that $G$ is almost simple.  The type of $G$ is $1$ (see also Example~\ref{example213}). If $H \subset G$ is a proper $\Delta$-subgroup of $G$, then there would exist a linear operator $g \in \UX[\Delta]$ such that $g(z)$ is not in the differential ideal $[\dd_tz-f(z)]$ generated by $\dd_tz-f(z)$ in $\UX\{z\}$ and such that $H \subset \{y \in G \ | \ g(y) = 0\}$.  We may assume that $g \in \UX[\dd_x]$. If  $g$ has order $d$, then for any $y \in H$, we have $k(y, \dd_xy, \dd^2_xy, \ldots ) = k(y, \dd_xy, \ldots , \dd_x^{d-1}y)$. Since $\dd_t y\in k(y, \dd_xy, \ldots , \dd_x^{d-1}y)$, this  field equals $k<y>$ which must therefore be of finite transcendence degree over $k$. Therefore $H$ has type $0$ and  $G$ is almost simple.} 
\end{ex}

%
%
It is not hard to see that a finite normal subgroup of a connected algebraic group must be central.  The next proposition is an analogous result for strongly connected $\Delta$-groups.
\begin{prop}\label{solidprop} Let $G$ be a strongly connected $\Delta$-group. Then every normal $\Delta$-subgroup of smaller type is central.
\end{prop}
We immediately have the following corollaries.
\begin{cor}\label{simplecor} Let $G$ be an almost simple $\Delta$-group. Any  proper normal $\Delta$-subgroup is central.
\end{cor}
\begin{proof} For any $\Delta$-group $G$, the strong identity component, $G_0$, is a normal $\Delta$-subgroup of the same type, and  so $G = G_0$, that is, $G$ is strongly connected. By definition  any proper normal $\Delta$-subgroup is of smaller type.
\end{proof}
If $G$ is an almost $k$-simple $\Delta$-$k$-group, the strong identity component is a normal $\Delta$-$k$-subgroup of the same type, and so, $G=G_0$.  By definition, any proper normal $\Delta$-$k$-subgroup has smaller type, and hence is central.

\begin{cor}\label{simplecor2} Let $G$ be an almost simple $\Delta$-group. Then $G/Z(G)$ is a simple $\Delta$-group.
\end{cor}
%
%
\noindent {\bf Proof  of Proposition \ref{solidprop}:} Let $N$ be a normal $\Delta$-subgroup of $G$.  We define a $\Delta$--rational map 
\[\alpha: G\times N \rightarrow N\]
by the formula 
\[\alpha(g,a) = gag^{-1}, \ \ g \in G, a \in N.\]
One checks that $\alpha(gh,a) = \alpha(g,\alpha(h,a))$ and $\alpha(1,a) = a$ and so $\alpha$ defines an action of the $\Delta$-group $G$ on $N$. For fixed $a \in N$, the map $\alpha_a(g)= gag^{-1}$ is a $\Delta$-rational map from $G$ to $N$ that is constant on the left cosets of $Z_G(a) = \{g \in G \ | \ gag^{-1} = a\}$.  Note that $Z_G(a)$ is a $\Delta$-subgroup of $G$ (\cite{kolchin_groups}, Section IV.4, Cor.~2(b)).\\[0.1in]
 Theorem 3, p.105 of \cite{kolchin_groups} implies that there is a $\Delta$-morphism $\beta:G/Z_G(a) \rightarrow N$ such that $\pi \circ \beta = \alpha_a$, where $\pi$ is the canonical quotient map. The image of $G/Z_G(a)$ under $\beta$ is the same as the image of $G$ under $\alpha_a$, namely  $Ga$. Let $V$ be the $\Delta$-closure of $Ga$. \\[0.1in]
Note that $\alpha_a(g) = \alpha_a(h)$ implies that $gag^{-1} = hah^{-1}$ so $h^{-1}g \in Z_G(z)$. Therefore $\beta$ is an injective $\Delta$-map. Corollary~\ref{taucor} implies that $\tau(G/Z_G(a)) \leq \tau(N) < \tau(G)$. Since $G$ is strongly connected, we have $Z_G(a) = G$.\hfill \QED

\subsubsection{Isogeny}
\begin{defin} 1) Let $G$, $H$ be strongly connected $\Delta$-$k$-groups. A $\Delta$-$k$-morphism $\phi:G\rightarrow H$ is an {\rm isogeny} if it is surjective and $\tau(\ker \phi) < \tau(G)$.\\[0.1in]
2) Two strongly connected $\Delta$-$k$-groups  $H_1,H_2 $  are {\rm $\Delta$-$k$-isogenous} if there exists a strongly connected $\Delta$-$k$-group $G $ and isogenies $\phi_i:G\longrightarrow H_i$  for $i = 1, 2$.
\end{defin}
The notion of  isogenous in the present context generalizes the notion of isogenous in the theory of algebraic groups (cf., \cite{rosenlicht_56}) where two connected algebraic groups $H_1, H_2$ are isogenous if there is an algebraic group $G$ and morphisms $\phi_i:G\rightarrow H_i$ for $i = 1,2$ with finite kernel.
\begin{ex} {\rm Let $G = \SL_6(\CX)$. The center of $G$ is isomorphic to the cyclic group of order $6$.  Let $C_2$ and $C_3$ be cyclic subgroups of this center of orders $2$ and $3$ respectively.  Let $H_1 = \SL_6(\CX)/C_2$ and $H_2= \SL_6(\CX)/C_3$ and let $\phi_i:G \longrightarrow H_i$ be the obvious projections.  These show $H_1$ and $H_2$ are isogenous in both the category of algebraic groups and the category of $\Delta$-groups.  Note that they are not isomorphic since their centers are finite groups of different orders.} \end{ex}

\begin{remarks}\label{remarkisog}{\rm 1.  If $G_{1}$ and $G_{2}$ are strongly connected $\Delta$-groups and there is an
isogeny $\phi :G_{1}\longrightarrow G_{2}$, then $G_{1}$ and $G_{2}$ are
isogenous.  However, the converse is false, as the above example shows.\\[0.1in]
2. If $\phi :G_{1}\longrightarrow G_{2}$ is an isogeny then equations (\ref{eqn1}) and (\ref{eqn2}) imply that  $\tau(G_1) = \tau(G_2)$ and $a_\tau(G_1) = a_\tau(G_2)$.\\[0.1in]
3. The composite of  isogenies of strongly connected $\Delta$-groups  is
an isogeny.\\[0.1in]
4. If $G_{1}$ and $G_{2}$ are strongly connected of the same type, then it is easy to
see that $G_{1}\times G_{2}$ is strongly connected. Suppose $G_{1}^{\prime }$ and $
G_{2}^{\prime }$ are strongly connected and $\phi _{i}:G_{i}\longrightarrow
G_{i}^{\prime }$ is an isogeny, $i=1,2$. Then, $\phi _{1}\times \phi
_{2}:G_{1}\times G_{2}\longrightarrow G_{1}^{\prime }\times G_{2}^{\prime }$, with kernel $\ker \phi _{1}\times \ker \phi _{2}$ is an isogeny. }
\end{remarks}
\begin{prop} Let $G_{1}$ and $G_{2}$ be strongly connected  $\Delta$-$k$-groups. The following are equivalent:

\begin{enumerate}
\item There exists a strongly connected $\Delta$-$k$-group $H$ and isogenies $\phi
_{1}:H\longrightarrow G_{1}$ and $\phi _{2}:H\longrightarrow G_{2}:$%
\begin{equation*}
\begin{array}{ccccc}
&  & H &  &  \\ 
& ^{\phi _{1}\swarrow } &  & ^{\searrow \phi _{2}} &  \\ 
G_{1} &  &  &  & G_{2}%
\end{array}%
.
\end{equation*}

\item There exists a strongly connected $\Delta$-$k$-group $K$ and isogenies $\psi
_{1}:G_{1}\longrightarrow K$ and $\psi _{2}:G_{2}\longrightarrow K$:
\begin{equation*}
\begin{array}{ccccc}
G_{1} &  &  &  & G_{2} \\ 
& ^{\psi _{1}\searrow } &  & ^{\swarrow \psi _{2}} &  \\ 
&  & K &  & 
\end{array}%
.
\end{equation*}
\end{enumerate}
\end{prop}
\begin{proof} Note that all the $\Delta$-groups and $\Delta$-homomorphisms are defined over $k$. Assume the first statement. Put $H_{1}=\phi _{1}(\ker \phi _{2})$,
and $H_{2}=\phi _{2}(\ker \phi _{1})$. Clearly, $H_{1}=\phi
_{1}(\ker \phi _{1}\ker \phi _{2})$, and $H_{2}=\phi _{2}(\ker
\phi _{1}\ker \phi _{2})$. Therefore,%
\begin{eqnarray*}
G_{1}/H_{1} &=&\phi _{1}\left( H\right) /\phi _{1}(\ker \phi
_{2})=\phi _{1}\left( H\right) /\phi _{1}(\ker \phi _{1}\ker
\phi _{2})=H/(\ker \phi _{1}\ker \phi _{2}) \\
G_2/H_2 &=&\phi _{2}\left( H\right) /\phi _{2}(\ker \phi _{2})=\phi
_{2}\left( H\right) /\phi _{2}(\ker \phi _{1}\ker \phi
_{2})=H/(\ker \phi _{1}\ker \phi _{2}).
\end{eqnarray*}%
Set $K=H/(\ker \phi _{1}\ker \phi _{2})$. Since $K$ is the image of
the strongly connected group $H$, it is strongly connected. Since $\tau (\ker \phi _{1}\ker
\phi _{2})<\tau \left( H\right) ,$ the projections $\psi
_{i}:G_{i}\longrightarrow G_i/H_i \simeq K$ are isogenies.\\[0.1in]
Now assume the second statement.  Let $G = \{ (g_1,g_2) \in G_1\times G_2 \ | \ \psi_1(g_1) = \psi_2(g_2) \}$, that is, the pull-back or fiber product. One sees that the natural projections $\phi_1:G \longrightarrow G_1$ and  $\phi_2:G \longrightarrow G_2$ are surjective. This implies that $\tau(G) \geq \tau(G_1)$. The kernel of the $\phi_i$'s is contained in $\ker \psi_1 \times \ker \psi_2$ and so $\tau(\ker \phi_1) \leq \tau(\ker \psi_1 \times \ker \psi_2) < \tau(G_1)$. Since $G_1 \simeq G/\ker(\phi_1)$ we have $\tau(G)= \tau(G_1)$ and $\phi_1$ is an isogeny. Similar reasoning shows that $\phi_2$ is an isogeny. \end{proof}
The following result is a reworking of Theorem 6, \cite{rosenlicht_56} in the context of $\Delta$-$k$-groups.
\begin{prop}\label{rosprop} Isogeny is an
equivalence relation on the set of strongly connected $\Delta$-$k$-groups. If $G_{1}$ and $%
G_{2}$ are isogenous strongly connected $\Delta$-$k$-groups, with $\tau =\tau \left(
G_{1}\right) =\tau \left( G_{2}\right) $, then there is a bijective
correspondence between the strongly connected $\Delta$-$k$-subgroups of $G_{1}$ with type $%
\tau $ and those of $G_{2}$ with type $\tau $.  Let $J_1$ and $K_1$ be $\Delta$-$k$-subgroups of $G_1$ corresponding to the $\Delta$-$k$-subgroups $J_2$ and $K_2$ of $G_2$. Then, $J_{1}\subseteq K_{1}$ if and
only if $J_{2}\subseteq K_{2}$, and $J_{1}\trianglelefteq K_{1}$ if and only
if $J_{2}\trianglelefteq K_{2}$. In this case, the strongly connected $\Delta$-$k$-groups $K_{1}/J_{1}\mbox{ and }
K_{2}/J_{2}$ are isogenous and have type $\tau$. If $G_{1},G_{2}$, and $G_{1}^{\prime },G_{2}^{\prime }$ are
pairs of isogenous $\Delta$-$k$-groups, then the $\Delta$-$k$-groups $G_{1}\times G_{1}^{\prime }$ and $%
G_{2}\times G_{2}^{\prime }$ are isogenous and have type $\tau$.
\end{prop}
\begin{proof}Since reflexivity and symmetry follow immediately from the definition, we
need only show that isogeny is transitive. So, let $%
G_{1},G_{2},H,G_{2},G_{3},H^{\prime }$ be strongly connected $\Delta$-$k$-groups with type $\tau$, and
let $\phi _{i}:H\longrightarrow G_{i}$, $i=1,2,$ $\phi _{i}^{\prime
}:H^{\prime }\longrightarrow G_{i},$ $i = 2,3,$  be isogenies. Then, the $\Delta$-$k$-group $H\times H^{\prime }$
is strongly connected, and $\phi _{2}\times \phi _{2}^{\prime }:H\times H^{\prime
}\longrightarrow G_{2}\times G_{2}$ is surjective. Let $K=\left( (\phi
_{2}\times \phi _{2}^{\prime })^{-1}\left( \text{Diag }G_{2}\times
G_{2}\right) \right)_0$, where $(\ldots)_0$ denotes the strong identity component. Then, $K$ is a strongly connected $\Delta$-$k$-group. Let $\pi _{H},\pi _{H^{\prime }},$
be the projections from $H\times H^{\prime }$ to $H$, and to 
$H^{\prime }$.  We have the following diagram

\begin{equation*}
\begin{array}{ccccccccc}
&  &  &  & K &  &  &  &  \\ 
&  &  & ^{\pi _{H}\swarrow } &  & ^{\searrow \pi _{H'}} &  &  &  \\ 
&  & H &  &  &  & H^{\prime } &  &  \\ 
& ^{\phi _{1}\swarrow } &  &^{\phi _{2}\searrow }  &  &^{\phi' _{2}\swarrow }&  & ^{\phi' _{3}\searrow } &  \\ 
G_{1} &  &  &  & G_2 &  &  &  & G_{3}%
\end{array}%
\end{equation*}%
We claim that the restrictions of the $\Delta$-$k$ homomorphisms $\pi _{H}$ and  $\pi _{H^{\prime }}$ to $K$ are isogenies. Since the composition of isogenies  is an isogenies, our claim implies that $G_1$ and $G_3$ are isogenous and so the isogeny relation is an equivalence relation.\\[0.1in]
To prove our claim, first note that (\ref{eqn1}) implies that $\tau \left( H\times H^{\prime }\right) =\tau \left( H\right) =\tau
\left( H^{\prime }\right) =\tau \left( G_{2}\times G_{2}\right) $, and $\tau
\left( \ker \left( \phi _{2}\times \phi _{2}^{\prime }\right) \right)
=\tau \left( \ker \phi _{2}\times \ker \phi _{2}^{\prime }\right)
<\tau \left( H\times H^{\prime }\right) $.  If $(h,h') \in K\cap \ker \pi_H$ then $h = 1$ and, since $\phi_2(h) = \phi_2'(h')$, we have $h' \in \ker \phi_2'$. Therefore $\ker \pi _{H}=1\times \ker \phi _{2}^{\prime
}$ and so $\tau \left( \ker \pi _{H}\right) <\tau \left( H\right) $. \
Similarly, $\tau \left( \ker \pi _{H^{\prime }}\right) <\tau \left(
H^{\prime }\right) $. We know that $\tau \left( K\right) =\tau \left(
G_{2}\right) =\tau \left( H\right) $. Since $\tau \left( \ker \pi
_{H}\right) <\tau \left( H\right) $, it follows that $a_{\tau }\left( \pi
_{H}\left( K\right) \right) =a_{\tau }\left( H\right) $. Since $H$ is
strongly connected, $\pi _{H}\left( K\right) =H$, and, therefore, $\pi _{H}$ is an
isogeny. Similarly, $\pi _{H^{\prime }}$ is an isogeny and this shows that the  isogeny relation is an equivalence relation on the set of strongly connected $\Delta$-$k$-groups.\\[0.1in]
Suppose we have the following isogeny diagram of strongly connected $\Delta$-$k$-groups:%
\begin{equation*}
\begin{array}{ccccc}
&  & H &  &  \\ 
& ^{\phi _{1}\swarrow } &  & ^{\searrow \phi _{2}} &  \\ 
G_{1} &  &  &  & G_{2}%
\end{array}%
.
\end{equation*}%
We define the following correspondence between strongly connected $\Delta$-$k$-subgroups of $%
G_{1}$ of type $\tau =\tau \left( G_{1}\right) $ and strongly connected $\Delta$-$k$%
-subgroups of $G_{2}$ of type $\tau =\tau \left( G_{2}\right) $:  Subgroups 
$J_{1}$ of $G_{1}$ and $J_{2}$ of $G_{2}$ correspond if there exists a strongly connected 
$\Delta$-$k$-subgroup $P$ of $H$ such that $\phi _{i}|_{P}$ maps $P$ to $J_i$ and  is an
isogeny, $i=1,2$. We claim that given $J_{1}$, there is a unique choice for $P$, namely,
the strong identity component of $\phi _{1}^{-1}(J_{1})$, which is, of course, a strongly connected $\Delta$-$k$-subgroup of $H$. First note that for any such $P$, Remark~\ref{remarkisog}.2 implies that $\tau(P) = \tau$ and $a_\tau(P) = a_\tau(J_1) = a_\tau(J_2)$. Furthermore any such $P$ must lie in $\phi _{1}^{-1}(J_{1})$ and, since it is assumed to be strongly connected, it must be contained in the strong identity component $(\phi _{1}^{-1}(J_{1}))_0$ of  $\phi _{1}^{-1}(J_{1})$. Since $\tau(\ker \phi_1) < \tau$ we have that $\tau(\ker\phi_1 \cap \phi^{-1}(J_1)) < \tau$ and so $\phi$ restricted to $(\phi _{1}^{-1}(J_{1}))_0$ is an isogeny.  In particular $\tau((\phi _{1}^{-1}(J_{1}))_0) = \tau $ and $a_\tau((\phi _{1}^{-1}(J_{1}))_0) = a_\tau(G_1)$.  Since $P \subset (\phi _{1}^{-1}(J_{1}))_0$ and both of these groups are strongly connected we have $P = (\phi _{1}^{-1}(J_{1}))_0$. A similar argument shows that $P = (\phi _{2}^{-1}(J_{2}))_0$. If follows easily that the correspondence is unique.\\[0.1in]
Now, we shall show that if $J_{i},K_{i},i=1,2,$ are corresponding strongly connected $%
\Delta $-$k$-subgroups of $G_{i}$ of type $\tau $, then%
\begin{equation*}
J_{1}\subseteq K_{1}\Longleftrightarrow J_{2}\subseteq K_{2}.
\end{equation*}
Let $P,Q,i=1,2,$ be the unique strongly connected $\Delta$-$k$-subgroups of $H$ of type $%
\tau $ such that $\phi _{i}\left( P\right) =J_{i}$ and $\phi
_{i}\left( Q\right) =K_{i}$. $P$ is the strong identity component of $\phi _{i}^{-1}\left(
J_{i}\right) $ and $Q$ is the strong identity component of $\phi _{i}^{-1}\left( K_{i}\right) $%
. $J_{i}\subseteq K_{i}$ implies that $\phi _{i}^{-1}\left(
J_{i}\right) \subseteq \phi _{i}^{-1}\left( K_{i}\right) $. $\tau
\left( P\right) =\tau \left( Q\right) =\tau .$ \ Since $P$ is a strongly connected $%
\Delta $-$k$-subgroup of $\phi _{i}^{-1}\left( K_{i}\right) $ of the same
type as the latter, it follows, as above, that $P$ is contained in its strong identity component $%
Q$. The reverse implication is proved the same way. Now suppose that $J_{1}$
is normal in $K_{1}$. Then, $\phi _{1}^{-1}\left( J_{1}\right) $ is
normal in $\phi _{1}^{-1}\left( K_{1}\right) $ and has the same type. By
Remark~\ref{remarksolid}.2, $P$ is normal in $\phi _{1}^{-1}\left( K_{1}\right) $%
. Therefore, $P$ is normal in $Q$. Since $\phi _{2}$ is surjective, $%
J_{2}$ is normal in $K_{2}$. To finish the proof, we need only show that the $\Delta$-$k$-groups $
K_{1}/J_{1}$ and $K_{2}/J_{2}$ are isogenous and have the same type $\tau$. First, note that Remark~\ref{remarksolid}.5 says that $K_i / J_i$ are strongly connected and have type $\tau$.
We now define a map $\phi:Q/P \rightarrow K_1/J_1$ by 
$\phi \left( q\text{ mod }P\right) =\phi _{1}\left( q\right) $ mod $%
J_{1}$ for any $q \in Q$. To see that this is well defined, let $q_{1},q_{1}^{\prime }\in Q$, and suppose $q_{1}$ mod $%
P=q_{1}^{\prime }$ mod $P$. Then, $\phi _{1}\left( q_{1}-q_{1}^{\prime
}\right) \in J_{1}$. Therefore, $\phi_1(q_{1})$ mod $J_{1}=\phi_1(q_{1}^{\prime })$ mod $%
J_{1}$. So, the map $\phi :Q/P\longrightarrow K_{1}/J_{1}$ is well
defined.  It is clearly a homomorphism. We claim that it is surjective. 
Let $k_{1}\in K_{1}.$  There exists $q\in Q$ such that $\phi _{1}\left(
q\right) =k_{1}$ and so  $\phi (q$ mod $P)=k_{1}$ mod $J_{1}$. A calculation shows that the kernel of $\phi$ is $[Q\cap\phi_1^{-1}(J_1)]/P$. 
The strong identity component $P$ of $\phi _{1}^{-1}\left( J_{1}\right) $ is contained in $%
Q\cap \phi _{1}^{-1}\left( J_{1}\right) $. Thus, $\tau \left( \lbrack
Q\cap \phi _{1}^{-1}\left( J_{1}\right) ]/P\right) <\tau \left( \phi
_{1}^{-1}\left( J_{1}\right) \right) =\tau \left( Q\right) =\tau \left(
Q/P\right) $ since $Q$ is strongly connected. So, $\phi :Q/P\longrightarrow
K_{1}/J_{1}$ is an isogeny. Similar reasoning shows that one can define a map  $\psi
:Q/P\longrightarrow K_{2}/J_{2}$ and that it too  is an isogeny. The above proof also shows that $J_1$ is defined over $k$ if and only if $J_2$ is defined over $k$. \end{proof}
An immediate corollary is:
\begin{cor} \label{roscor}Let $G_{1}$ and $G_{2}$ be isogenous strongly connected $\Delta$-$k$-groups. Either both
are almost simple or neither is.
\end{cor}
\begin{proof}
Suppose $G_{1}$ is almost simple.  Let $N$ be a normal $\Delta $-subgroup
of $G_{2}$.  There exists a finitely generated $\Delta $-$k$-subfield $%
k^{\prime }$ of $\mathbb{U}$ such that $N$ is defined over $k^{\prime }$. 
Since $k^{\prime }$ is finitely $\Delta $-generated over $k$, $\mathbb{U}$
is universal over $k^{\prime }$.  Moreover, as we remarked in Section~\ref{subsec2.1},
the type of $G_{2}$ is invariant under the choice of field of definition. 
If the type of $N$ is $\tau =\tau \left( G_{2}\right) =\tau \left(
G_{1}\right) $, Proposition~\ref{rosprop} gives us a proper normal $\Delta $-$%
k^{\prime }$- subgroup of $G_{1}$ of type $\tau $, contradicting the almost
simplicity of $G_{1}$.
\end{proof}
Another corollary is
\begin{cor} \label{corcom}Let $G_{1}$ and $G_{2}$ be isogenous strongly connected $\Delta$-groups. Either both
are commutative or neither is.
\end{cor}
The proof of Corollary~\ref{corcom} depends on the following concept and lemma.
\begin{defin}\label{defcom} Let $G$ be a $\Delta$-$k$-group. The {\rm differential commutator group} $D_\Delta(G)$ is the smallest $\Delta$-subgroup of $G$ containing the commutator subgroup of $G$. It is defined oer $k$.
\end{defin}
An example of Cassidy ({\it cf.},~p.111 of \cite{kolchin_groups}) shows that $D_\Delta(G)$ may be strictly larger than the commutator subgroup of $G$.
\begin{lem} \label{lemcom} Let $G$ be a $\Delta$-$k$-group. If $G$ is strongly connected and not commutative, then $\tau(D_\Delta(G)) = \tau(G)$.
\end{lem}
\begin{proof} Let $a$ be in $G$.  Let $c_a:G\rightarrow G$ be the $\Delta$-map that sends $x$ to $axa^{-1}x^{-1}$. If $D_\Delta(G)$ has smaller type than $G$, then Proposition~\ref{solidprop} implies that $D_\Delta(G)$ is contained in the center $Z(G)$. In this case we have $axya^{-1}y^{-1}x^{-1} = axa^{-1}(aya^{-1}y^{-1})x^{-1} = (axa^{-1}x^{-1})(aya^{-1}y^{-1})$ and therefore $c_a$ is a $\Delta$-homomorphism of $G$ into $D_\Delta(G)$. Since $G$ is strongly connected, $c_a(G) = \{1\}$.  Thus, $a \in Z(G)$. It follows that $G$ is commutative.\end{proof}
%
\noindent{\bf Proof of Corollary~\ref{corcom}:}
Let $K$ be a $\Delta$-$k$-group such that for  $G_{1},G_{2},$ and $K$ we  have the commutative
diagram:%
\begin{equation*}
\begin{array}{ccccc}
G_{1} &  &  &  & G_{2} \\ 
& ^{\psi _{1}\searrow } &  & ^{\swarrow \psi _{2}} &  \\ 
&  & K &  & 
\end{array}%
\end{equation*}%
with isogenies $\psi_i$.
Assume $G_{1}$ is commutative. Then, $K$ is commutative.  Suppose $G_{2}$ is not commutative. Lemma~\ref{lemcom} implies that $\tau (D\left( G_{2}\right) )=\tau \left( G_{2}\right) =\tau \left(
K\right) $. Since $\psi_2$ is an isogeny,  $\tau(\psi_2\left( D_\Delta\left( G_{2}\right) \right) =\tau(K)$ as well. Since $\psi_2(D_\Delta(G_2))\subset  D_\Delta(K)$ we have that $D_\Delta(K) \neq \{1\}$, a contradiction.\hfill \QED\\[0.2in]
 We give two examples of  infinite families of isogenous, pairwise nonisomorphic $\Delta$-groups. 
\begin{ex}\label{Gaex}{\rm Let $\Delta = \{ \delta\}$ and $\frakU$ be a $\Delta$-universal field with $\Delta$-constants $C$.  We shall exhibit an infinite number of almost simple $\Delta$-groups that are isogenous  to $\Ga(\frakU)$ but are pairwise nonisomorphic. For each $n \in \NX$, let 
\[G_n = \{g(a,b) = \left(\begin{array}{ccc} a & 0 & 0 \\ 0 & 1 & b \\ 0&0&1\end{array}\right) \ | \ a \in \Gm(\UX), b \in \Ga(\UX), (\delta a) a^{-1} = \delta^n(b)\}.\] 
The map $\alpha:G_n \rightarrow \Ga(\frakU)$ given by $\alpha(g(a,b)) = (\delta a)a^{-1}$ is a surjective homomorphism with kernel
\[K_n = \{ g(a,b) \ | \ \delta a = 0, \delta^n(b) = 0 \}.\]
Note that $G_n$ has differential type 1 and that $K_n$ has differential type $0$; so $\alpha$ is an isogeny and therefore $G_n$ is isogenous to $\Ga(\frakU)$.\\[0.1in]
We shall now show that $G_n$ is strongly connected. Let $(G_n)_0$ be the strong identity component of $G_n$ and let $\pi_1(g(a,b)) = a$ and $\pi_2(g(a,b)) = b$ be the projections of $G_n$ onto $\Gm(\UX)$ and $\Ga(\UX)$ respectively. Since these latter groups are strongly connected, these projections are surjective when restricted to  $(G_n)_0$ as well. Using $\pi_1$ we see that for any $g(a,b) \in G_n$ there exists $b_1 \in \Ga(\UX)$ such that $g(a,b_1) \in (G_n)_0$ and therefore  $g(a,b) = g(1,b-b_1)g(a,b_1)$. Note that $\delta^n(b-b_1) = 0$.  This implies that the group $H_n = \{ b \in \Ga(\UX) \ | \ \delta^n(b) = 0\}$ is mapped surjectively onto $G_n/(G_n)_0$ via the map $b \mapsto g(1,b) \mod (G_n)_0$.  Using $\pi_2$, a similar argument shows that  the map $c \mapsto g(c,0) \mod (G_n)_0$ maps $\Gm(C) = \{ c\in \UX \ |  \ c\neq 0, \delta(c) = 0\}$ surjectively onto $G_n/(G_n)_0$. Therefore $G_n/(G_n)_0$ is isomorphic to a quotient of $H_n$ and a quotient of $\Gm(C)$. Taking a fundamental system of solutions of the differential equation $%
\delta ^{n}=0$, we see that $H_{n}$ is isomorphic as a $\Delta $-group to $%
\left( C^{n},+\right) $.  Therefore, any quotient $\Delta $-group is
torsion-free. This implies that $G_{n}/\left( G_{n}\right) _{0}$ is
torsion-free.  But, the $\Delta $-group $G_{n}/\left( G_{n}\right) _{0}$ is
also a quotient of $\mathbb{G}_{\text{m}}\left( C\right) $, whose torsion
group is Kolchin dense. \ Therefore, either the torsion group of $%
G_{n}/\left( G_{n}\right) _{0}$ is Kolchin dense, or $G_{n}/\left(
G_{n}\right) _{0}$ is trivial. \ Since the latter must hold, $G_{n}$ is
strongly connected.\\[0.1in]
Since $\Ga(\frakU)$ is almost simple, Corollary~\ref{roscor} implies that $G_n$ is almost simple.  \\[0.1in]
We now claim that for $n \neq m, n,m \in \NX$,  $G_n$ and $G_m$ are not $\Delta$-isomorphic.  To see this, note first that the set of unipotent matrices in any $G_n$ is precisely $\{g(1,b) \ | \ \delta^n(b) = 0\}$.  This is a $\Delta$-group of differential type $0$ and typical differential dimension equal to  $n$.  The property of being unipotent is preserved under $\Delta$-isomorphism (\cite{cassidy1}, Proposition 35) so our claim is proved. Note also that there does not exist an isogeny from $\Ga (\UX)$ onto $G_n$, and, therefore, this example does not contradict Example 2.11.3.}\end{ex}
\begin{ex}\label{cartanexple} {\rm  Let $\mathbb{C}$ be the field of complex numbers, and let $k$ be the
subfield $\mathbb{C}\left( x,t,\{e^{c^{2}t},e^{icx}\}_{c\in \mathbb{C}}\right) $
of the field $\mathbb{M}$ of functions meromorphic on $\mathbb{C}^{2}$. \
Let $\Delta =\left\{ \delta _{x},\delta _{t}\right\} $, where $\delta _{x}$
acts on $\mathbb{M}$ as $\frac{\partial }{\partial x}$ and $\delta _{t}$
acts as $\frac{\partial }{\partial t}$ . \ Then, $\mathbb{M}$ is a $\Delta $%
-field, and $k$ is a $\Delta $-subfield of $\mathbb{M}$. \ Let $\mathbb{U}$
be a $\Delta $-field universal over $k$, and with field $C$ of constants, as
usual. \ We do not assume that $\mathbb{M}$ is a subfield of $\mathbb{U}$.  Let \[H = \{ u \in \Ga(\frakU) \ | \ \dd_tu -\dd_x^2 u = 0 \} \] be the subgroup of $\Ga(\frakU)$ defined by the Heat Equation. We shall exhibit an infinite family of almost simple $\Delta$-groups that are isogenous to $H$ but not $\Delta$-isomorphic to $H$.  In particular, we shall show that if $\lambda \in \frakU$ is any nonzero solution of $\delta_t\lambda + \delta_x^2\lambda = 0$, then $H$ is isogenous to \[H_\lambda = \{u \in \Ga(\frakU) \ | \ \dd_tu - \dd_x^2u + 2\frac{\dd_x\lambda}{\lambda} \dd_xu = 0 \}\]
and that there are an infinite number of $\lambda \in k$ whose corresponding groups are pairwise not $\Delta$-isomorphic.\\[0.1in]
The groups $H$ and $H_\lambda$ have differential type $1$ and, from Example~\ref{evolex}, we know that they are almost simple.  An elementary calculation shows that  the map $\phi(u) = (1/\lambda) \dd_xu$ is a homomorphism from $H_\lambda$ onto $H$. The kernel of $\phi$ is $\Ga(C)$. Therefore these groups are isogenous.\\[0.1in]
For any nonzero $c\in \CX$,  \[\lambda_c = e^{c^2t}\cos cx\] is a solution in $k$ of $\dd_t\lambda+\dd_x^2\lambda = 0$.  We define
\[H_c = H_{\lambda_c} = \{ u \in \Ga(\frakU) \ | \ \dd_tu = \dd_x^2u + 2(c\tan cx) \dd_x u \} .\]
Each of these groups is isogenous to $H$ so they are all pairwise isogenous. We shall show that for nonzero $c,d \in C$ with $c/d \notin \QX$, the groups $H_c$ and $H_d$ are not isomorphic.
 The proof of this claim follows from the special nature of the differential coordinate rings  $\frakU\{H_c\}_\Delta$ and $\frakU\{H_d\}_\Delta$ of the groups.  \begin{eqnarray*}
\mathbb{U}\left\{ H_{c}\right\} _{\Delta } &=&\mathbb{U}\left\{ z\right\}
_{\Delta },\text{ where }\delta _{t}z=\delta _{x}^{2}z+2\left( c\tan
cx\right) \delta _{x}z, \\
\mathbb{U}\left\{ H_{d}\right\} _{\Delta } &=&\mathbb{U}\left\{ y\right\}
_{\Delta },\text{ where }\delta _{t}y=\delta _{x}^{2}y+2\left( d\tan
dx\right) \delta _{x}y.
\end{eqnarray*}%
Note that both groups are defined over $k$. \ The form of the defining
differential equations imply that
  \[\frakU\{H_c\}_\Delta = \frakU\{H_c\}_{\{\delta_x\}} = \frakU[z, \dd_x z, \ldots , \dd^n_x z, \ldots ] \mbox{ and } \frakU\{H_d\}_\Delta = \frakU\{H_d\}_{\{\delta_x\}} = \frakU[y, \dd_x y, \ldots , \dd^n_x y, \ldots ].\] where $y$ and its $\dd_x$-derivatives are algebraically independent over $\frakU$ and a similar statement is true for $z$. As  $\dd_x$-rings these two rings are $\dd_x$-isomorphic to a $\dd_x$-differential polynomial ring $\frakU\{Y\}_{\dd_x}$ and this will be a key fact in the proof of our claim.   What distinguishes these rings is their $\dd_t$-structure, which is given by the defining equations of the groups. The derivation operator $\delta_t$ acts as an "evolutionary derivation" on the $\delta_x$-differential ring.\\[0.1in]
Let us assume (with the aim of arriving at a contradiction) that $H_c$ and $H_d$ are isomorphic as $\Delta$-groups and let $\psi:H_d \rightarrow H_c$ be a differential isomorphism.  This isomorphism induces a $\Delta$-isomorphism $\psi^*:\frakU\{H_c\}_\Delta\rightarrow \frakU\{H_d\}_\Delta$, defined by the formula 
$\psi^* (f)= f$ composed with $\psi$.  Note that $\psi^*$ restricted to $U$ is the identity automorphism. Identifying both  $\frakU\{H_c\}$  and $\frakU\{H_d\}_\Delta$ with the $\dd_x$-ring $\frakU\{Y\}_{\{\dd_x\}}$, the map $\psi^*$ induces a $\dd_x$-automorphism of $\frakU\{Y\}_{\{\dd_x\}}$ and {\it a fortiori} of the quotient field $\frakU<Y>_{\{\dd_x\}}$.  The $\delta_x$-U-automorphisms of $\UX<Y>_{\delta_x}$ have been classified by Ritt (as part of his differential L\"uroth theorem) in his book (\cite{RITT}, p.~56).  He shows that any $\dd_x$-automorphism $\sigma$ of this latter field must be of the form \[\sigma(Y) = \frac{aY+b}{cY+d}\] where $a,b \in \frakU$.  It is not hard to see that if $\sigma$ furthermore maps $\frakU\{Y\}_{\{\dd_x\}}$ to itself, we must have that $\sigma(Y) = aY+b$ for some $a,b \in \frakU$. In particular, $\psi^*$ must be of the form $\psi^*(z) = ay + b$ for some $a,b \in \frakU$ and, since $\psi$ is assumed to be a homomorphism of additive groups  as well, we must have $b=0$ and so $\psi^*(z) = ay$. The following calculation will show that this leads to a contradiction. \\[0.1in] 
We have that
\begin{eqnarray*}
\dd_tz &=& \dd_x^2z + 2(c\tan cx) \dd_xz\\
\dd_ty &=& \dd_x^2y + 2(d\tan dx) \dd_xy
\end{eqnarray*}
Note that $c$ goes with the domain of the comorphism and $d$ with its range. Calculating $\dd_t(\psi^*(z))$ and $\psi^*(\dd_tz)$ we have
\begin{eqnarray*}
\dd_t(\psi^*(z)) &=& a \dd_x^2y + (2ad (\tan dx)) \dd_xy + \dd_t a y \\ 
\psi^*(\dd_tz) &=& a \dd_x^2y + (2\dd_xa + 2ac(\tan cx)) \dd_xy + (\dd_x^2a + 2 c(\tan cx)\dd_xa)y  
\end{eqnarray*}
Since $\dd_t(\psi^*(z))=\psi^*(\dd_tz)$ and $ y, \dd_x y, \dd_x^2y $ are algebraically independent, we must have
\begin{eqnarray}
ad (\tan dx) & = & \dd_xa + ac(\tan cx) \label{eqn10}\\
\dd_t a & = & \dd_x^2a + 2 c(\tan cx)\dd_xa \label{eqn11}
\end{eqnarray}
Equation (\ref{eqn10}) implies that 
\[a = \alpha \frac{\cos cx}{\cos dx}\]
where $\alpha$ is in the field of constants of $\delta_x$ in $\UX$.  Inserting this expression into equation~(\ref{eqn11}) we have 
\begin{eqnarray}
\frac{\dd_t\alpha }{\alpha}  & = &  c^2 + \frac{2d^2}{\cos^2dx} - d^2 - \frac{2c^2}{\cos^2 cx}.\label{eqn12}
\end{eqnarray}
Since $\dd_x(\frac{\dd_t\alpha }{\alpha}) = 0$, we must have that the right-hand side of equation~(\ref{eqn12}) is some constant $\gamma \in C$. A linear disjointness argument implies that $\gamma$ is in the field of constants of $k$.  In particular, it is a complex number.  So, the terms of equations (\ref{EQN00}) and (\ref{EQN000}) are complex functions.  Evaluating the right-hand side at $x =0$ and $x=\frac{2\pi}{c}$, we have
\begin{eqnarray}
d^2 - c^2 & = & \gamma \label{EQN00}\\
(\frac{2}{\cos^2(2d\pi/c)} - 1)d^2 -c^2 & = & \gamma \label{EQN000}
\end{eqnarray}
Note that $\cos(2d\pi/c)\neq 0$ since we are assuming that $d/c$ is neither an integer nor half an integer. From these equations, we can conclude that $\cos^2(2d\pi/c) = 1$. Therefore, $\cos (2 d\pi/c) = 1$ or $-1$.  So, properties of the complex cosine function tell us that 
    $2 \pi d/c$ equals  $2n \pi$ or $(2n+1) \pi$.  So, $d/c$ is either an integer or $d/c$ is half an integer, a contradiction.   Therefore, for any set of constants $S$, whose members are pairwise independent over the integers, this yields a family $\{H_c\}_{c\in S}$ of isogenous  $\Delta$-subgroups of $\Ga$ that are pairwise nonisomorphic.\\[0.1in]

%
%
\noindent We now turn to another choice of $\lambda$: $\lambda = x$. This situation   was already considered by Cartan (\cite{cartan09}, p.145-146).  Letting
\[H_x = \{ u \in \Ga(\frakU) \ | \ \dd_tu = \dd_x^2u  - \frac{2}{x} \dd_x u \}, \]
Cartan states (without proof) that $H$ and $H_x$  are not ``isomorphes holo\'edriques''. We shall show that $H$ and $H_x$ are not $\Delta$-isomorphic. \\[0.1in]
%
Let us assume  that $H$ and $H_x$ are isomorphic as $\Delta$-groups and let $\psi:H \rightarrow H_x$ be a differential isomorphism.  
Let \[\frakU\{H\}_\Delta = \frakU[y, \dd_x y, \ldots , \dd^n_x y, \ldots ] \mbox{ and } \frakU\{H_x\}_\Delta = \frakU[z, \dd_x z, \ldots , \dd^n_x z, \ldots ]\] where $y$ and its $\dd_x$-derivatives are algebraically independent over $\frakU$ and a similar statement is true for $z$. As above, we have a differential isomorphism $\psi^*:\frakU\{H_x\}_\Delta \rightarrow \frakU\{H\}_\Delta$ and we may assume that $\psi^*(z) = ay$. The following calculation will show that this leads to a contradiction. 
%
 Since $\psi^*$ is also a $\Delta$-$\UX$-isomorphism, we must have that $\psi^*(\dd_t z) = \dd_t(\psi^*(z))$.  We have
 \begin{eqnarray}
 \psi^*(\dd_tz) & = &  a\dd_x^2y +(2\dd_xa-\frac{2}{x} a)\dd_xy +(\dd_x^2a - \frac{2}{x}\dd_xa)y \label{careqn1}\\
\dd_t(\psi^*(z)) & = & a\dd_x^2y + (\dd_t a) y.\label{careqn2}
 \end{eqnarray}  
 Since $\dd_x^2y, \dd_xy$ and $y$ are algebraically independent over $\frakU$, comparing coefficients of like terms in (\ref{careqn1}) and (\ref{careqn2}), we have 
 \begin{eqnarray}
 x\dd_xa - a & = & 0 \label{careqn3}\\
 x\dd_x^2a-2\dd_xa -x\dd_ta & = & 0.\label{careqn4} \end{eqnarray}  
 Equation (\ref{careqn3}) implies that $a = cx$ where $c$ is a $\delta_x$-constant in $\UX$. $\dd_xc = 0$.  Therefore $x\dd_x^2 a = 0$.  This, combined with equation (\ref{careqn4}), implies that $0 = 2\dd_x a + x\dd_ta = 2c + (\dd_tc)x^2$. Therefore $0=\dd_x(2c+(\dd_tc)x^2) = 2\dd_tc x$, so $\dd_tc=c=0$.  We therefore have that $a=0$.  This contradicts the fact that $\psi^*$ is an isomorphism and completes the proof that $H$ and $H_x$ are not $\Delta$-isomorphic.\\[0.1in]
The Heat Equation $\dd_ty = \dd_x^2y$ and the related equations $\dd_tz = \dd_x^2z - \frac{\dd_t \lambda}{\lambda} \dd_xz$ arise from the study of conservation laws and generalized symmetries (cf., \cite{bluman_kumei}, \cite{popkuniva}), which we now briefly describe.  Let $L$ be a homogeneous linear $\Delta$-polynomial.  A vector $ {\bf F} = (T(y), X(y))$ of differential polynomials is called a {\it conserved vector} if $\nabla {\bf F}(u) = \dd_tT(u) + \dd_xX(u) = 0$ for all solutions of $L(y) = 0$.  The equation $\dd_tT(u) + \dd_xX(u) = 0$ is called a {\it conservation law} for $L(y) = 0$.  If $L$ is the Heat Equation, then it follows from \cite{popkuniva} that the conservation laws for $L(y)=0$ are generated (as a $C$-vector space) by conservation laws corresponding to conserved vectors of the form 
\begin{eqnarray}{\bf F}& =& (\lambda y, -\lambda \dd_xy+(\dd_x\lambda)y) \label{eqn20}
\end{eqnarray} where $\lambda$ satisfies $\dd_y\lambda + \dd_x^2 \lambda = 0 $. In \cite{bluman_kumei}, Bluman and Kumei introduce the notion of  a {\it potential symmetry} corresponding to each conservation law.  The potential symmetry corresponding to the conserved vector (\ref{eqn20}) is given by the system
\begin{eqnarray*}
\dd_xz & = & \lambda y\\
\dd_tz & = & \lambda \dd_xy - (\dd_x\lambda)y.
\end{eqnarray*}
This latter system defines a $\Delta$-group $G \subset \Ga\times \Ga$ that has been extensively studied for $\lambda = 1$ and $\lambda = x$.  Letting $\pi_1$ and $\pi_2$ be the projections of $\Ga\times \Ga$ onto the first and second coordinates of $(y,z)$, one can show that $\pi_1(G) = H$, $\pi_2(G) = H_\lambda$ and the $\pi_i$ are isogenies.  This gives another proof that the groups $H$ and $H_\lambda$ are isogenous. \\[0.1in]
%
 Further studies of the algebraic and model-theoretic structure of  the solutions of the Heat Equation and related differential algebraic groups can be found in the thesis \cite{suer07} of Sonat Suer.}\end{ex}

\subsection{Jordan-H\"older Theorem}
We can now state and prove the following analogue of the Jordan-H\"older Theorem (cf., Theorem 6, \cite{rosenlicht_56}).
\begin{thm} \label{jhthm} Let $G$ be a strongly connected $\Delta$--group with $\tau(G) = \tau \neq -1$ (i.e., $G \neq \{1\}$).  There exists a normal sequence $G = G_0 \rhd G_1 \rhd  \ldots \rhd  G_r = \{1\}$ of $\Delta$-subgroups such that for each $i = 0,\ldots ,r-1$
\begin{enumerate}
\item $G_i$ is a strongly connected group with $\tau(G_i) =\tau $,
\item $a_\tau(G_0) > a_\tau(G_1) > \ldots > a_\tau(G_{r-1})$, and
\item $G_i/G_{i+1}$ is almost simple of type $\tau$ and typical differential dimension $a_\tau (G_i) - a_\tau (G_{i+1})$.
\end{enumerate}
If $G =H_0\rhd H_1\rhd\ldots \rhd H_s = \{1\}$ is another normal sequence as above, then $r = s$ and there exists a permutation $\pi$ of $\{0,\ldots , r\}$ such that $G_{\pi(i)}/G_{\pi(i)+1}$ is $k$-isogenous to $H_i/H_{i+1}$.
\end{thm}
\begin{proof} If $G$ is not almost simple, there exists a normal proper $\Delta$-subgroup $H$ with $\tau(H) = \tau(G) = \tau$.  Since $G$ is strongly connected, we have that $a_\tau(H) < a_\tau(G)$. Among such subgroups let $H$ have largest $a_\tau(H)$ and let $G_1$ be  the strong identity component of $H$.   Note that, by definition, $\tau(G_1) = \tau(H)$. Since $G_1$ is a characteristic subgroup of $H$, it is normal in $G$.  We have that $a_\tau(G) = a_\tau(G/G_1) +a_\tau(G_1)$ so $a_\tau(G/G_1) \neq 0$.  Therefore $\tau(G/G_1) = \tau$.  For any normal $\Delta$-subgroup $K$ of $G$ with $K \supset G_1$, maximality of $a_\tau(H) = a_\tau(G_1)$ implies that $a_\tau(K) = a_\tau(G_1)$ so $\tau(K/G_1) < \tau$.  Therefore $G/G_1$ is almost simple. The type of $G/G_1$ is $\tau$ and its typical dimension is $a_\tau (G) - a_\tau (G_1)$. Proceeding by induction on $a_\tau(G_i)$ yields the required sequence of $\Delta$--subgroups.\\[0.1in]
To prove the last sentence in the theorem we proceed as in the proof of the Jordan-H\"older Theorem as given in (\cite{LANG}, Ch.~1,\S 3). For each $i = 0, \ldots r-1, \ j = 0, \ldots , s-1$ let  
\[ G_{i,j} = G_{i+1}(H_j\cap G_i) \ . \]
The $G_{ij}$ give a refinement of the sequence of $G_i$'s:
\[G = G_{0,0} \rhd G_{0,1}\rhd \ldots \rhd G_{0,s-1}\rhd G_1 = G_{1,0}\rhd G_{1,1} \rhd \ldots \rhd \{0\}\]
Similarly, define for each $j = 0, \ldots , s-1, \ i = 0, \ldots , r$,
\[H_{j,i} = H_{j+1}(G_i\cap H_j) \ . \]
This yields a refinement of the $H_j$.  By Lemma 3.3 of \cite{LANG} we have that 
\[G_{i,j}/G_{i,j+1} \mbox{ is isomorphic to } H_{j,i}/H_{j,i+1}  \]
(the isomorphism obviously being an isomorphism of $\Delta$-$k$-groups.)  We now claim that for any $i = 0, \ldots, r-1$ there is at exactly one $j, 0\leq j \leq s-1$ such that $\tau(G_{i,j}/G_{i,j+1}) = \tau$ and that for this value of $j$ we have that $G_{i,j}/G_{i,j+1}$ is isogenous to $G_i/G_{i+1}$ (and a similar statement will hold for the sequence of $H_{j,i}$). This and the previous sentence then imply that $r=s$ and the final statement of the theorem is true.\\[0.1in]
For each $i = 0 ,\dots , r-1$,  have that 
\[a_\tau(G_i) = a_\tau(G_{i,0}/G_{i,1})+ \ldots + a_\tau(G_{i,s-1}/G_{i+1,0})+ a_\tau(G_{i+1})  \]
(recall that $G_{i+1,0} = G_{i+1}$). Since $a_\tau(G_i) >a_\tau(G_{i+1})$  we have that some \linebreak $a_\tau(G_{i,j}/G_{i,j+1})$ $ \neq 0$ and so $\tau(G_{i,j}/G_{i,j+1}) = \tau$.  Equation (\ref{eqn1}) furthermore implies that \linebreak $\tau(G_{i,j}/G_{i+1}) = \tau$.  Let $j$ be the smallest integer such that  $a_\tau(G_{i,j}/G_{i,j+1})$ $ \neq 0$. \\[0.1in]
We will show that $G_i = G_{i,0} = G_{i,1} = \ldots = G_{i,j}$. Note that $\tau = \tau(G_{i,j}/G_{i,j+1})  \leq \tau(G_{i,j}/G_{i+1}) \leq \tau(G_{i}/G_{i+1}) = \tau$.  Therefore $\tau(G_{i,j}/G_{i+1}) = \tau$.  Furthermore, $\tau = \tau(G_{i}/G_{i+1}) \geq \tau(G_{i,t}/G_{i+1}) \geq \tau(G_{i,j}/G_{i+1}) = \tau$  so $\tau(G_{i,t}/G_{i+1})  = \tau$ for any $t = 0, 1, \ldots , j-1$.  Since $a_\tau(G_{i,t}/G_{i,t+1}) =0$ for $t = 0, 1, \ldots , j-1$, we have $\tau(G_{i,t}/G_{i,t+1}) < \tau$.  Since $G_{i}/G_{i+1}$ is strongly connected, $G_{i,1}/G_{i+1} \lhd G_i/G_{i+1}$ and 
$\tau( (G_{i,0}/G_{i+1}) /(G_{i,1}/G_{i+1}) ) =\tau (G_{i,0}/G_{i,1}) < \tau$, we have $G_{i,0} = G_{i,1}$.  Continuing in a similar fashion, we see $G_{i,0} = \ldots = G_{i,j}$.\\[0.1in]
From this discussion we can conclude that $G_{i,j} = G_i$ and  $G_{i,j+1} \lhd G_i.$ Let $\pi: G_i/G_{i+1} \longrightarrow G_{i,j}/G_{i,j+1}$ be the canonical projection.    The kernel is a proper normal subgroup so, again by almost simplicity, it must have smaller type. Therefore $G_i/G_{i+1}$ is isogenous to $G_{i,j}/G_{i,j+1}$. \\[0.1in]
We now turn to the uniqueness claim. Assume that $\tau(G_{i,t}/G_{i,t+1}) = \tau$ for some $t>j$.  We would then have $\tau(G_{i,t}/G_{i+1}) = \tau$ and so $\tau(G_{i,j+1}/G_{i+1}) = \tau$. Since $G_i/G_{i+1}$ is almost simple and $G_{i,j+1}\lhd G_{i,j} = G_i$, we must have $G_{i,j+1} = G_{i,j}$. contradicting $a_\tau(G_{i,j}/G_{i,j+1}) \neq 0$.\end{proof}

\begin{remarks} {\rm 1.~As is clear from  the above  proof the analogue of the results of Zassenhaus  and Schreier (cf., \cite{LANG}, Chapter I,\S3, Lemma 3.3 and Theorem3.4) stating that two normal sequences have isomorphic refinements is true in the context of $\Delta$-groups.  This follows from the fact that the  isomorphism theorems are true in this context.\\[0.1in]
2.~A parallel theorem, in which the groups in the subnormal series and the quotients are $\Delta$-$k$-groups, and the quotients are almost $k$-simple, holds for strongly connected $\Delta$-$k$-groups.}\end{remarks}
%
\begin{cor} Let $G$ and $H$ be isogenous strongly connected $\Delta$-groups with  normal sequences $G = G_0 \rhd G_1 \rhd  \ldots \rhd  G_r = \{1\}$ and $H =H_0\rhd H_1\rhd\ldots \rhd H_s = \{1\}$ as in Theorem~\ref{jhthm}. Then $r=s$ and there exists a permutation $\pi$ of $\{0,\ldots , r\}$ such that $G_{\pi(i)}/G_{\pi(i)+1}$ is $k$-isogenous to $H_i/H_{i+1}$.
\end{cor}
\begin{proof} Immediate from Theorem~\ref{jhthm} and Proposition~\ref{rosprop}. \end{proof}
\begin{remarks} {\rm 1. One can consider groups $G$ that are not strongly connected and apply Theorem~\ref{jhthm} first to the strong identity component $G_0$, then  to the strong identity component of $G/G_0$, etc.  This  will yield a  more complete Jordan-H\"older Theorem. \\[0.1in] 
 2.  In Example~\ref{cartanex} below, we show that the uniqueness up to isogeny in Theorem~\ref{jhthm} cannot be strengthened to give uniqueness up to isomorphism.\\[0.1in]
3. In \cite{Baud90}, Baudisch presents a Jordan-H\"older Theorem for superstable groups $G$.  He shows that such a group has a sequence of definable subgroups  $(1) \lhd H_0 \lhd H_1 \lhd \ldots \lhd H_r \lhd G$ such that $H_{i+1}/H_i$ is infinite and either abelian or simple modulo a finite center and $H_r$ is of finite index in $G$.  Although this is of a similar nature as our result, his result does not address the structure of abelian groups  nor does he address the question of uniqueness. A key step in understanding the relationship between Baudisch's result and ours would be to understand the relationship between U-rank and the type and typical differential transcendence degree (cf., \cite{suer07}).}\end{remarks}

\begin{ex} \label{example213} {\rm  We now return to the example of Landau mentioned in the Introduction:
\begin{eqnarray} L &=& \dd_x^3  + x\dd_x^2\dd_y + 2\dd_x^2 + 2(x+1)\dd_x\dd_y +\dd_x +(x+2)\dd_y\label{blumex}\\
& = & (\dd_x+1)(\dd_x+1)(\dd_x +x\dd_y) \label{factor1} \\ 
 & = & (\dd_x^2 + x\dd_x\dd_y + \dd_x + (x+2)\dd_y)(\dd_x+1) \label{factor2}\ .
 \end{eqnarray}   We shall show how this example fits into the above discussion. \\[0.1in]
 Let $\Delta = \{\dd_x, \dd_y\}$ and $k$ a differentially closed $\Delta$-field. Recall from Section~\ref{subsec2.1} that  the gauge of a $\Delta$-$k$-variety $X$ is the pair $(\tau(X), a_\tau(X))$ and order these lexicographically. Let $R$ be any nonzero homogeneous linear differential polynomial  in $\dd_x$ and $\dd_y$ of order $d$ and let $G_R \subset k$ be the solutions of $R=0$.  In \cite{DAAG}, Kolchin gives a general method for computing differential dimension polynomials $\omega_{\eta/k}(s)$ in terms of a characteristic set for the defining ideal $I$ of $\eta$ over $k$. When this ideal is generated by linear homogeneous differential polynomials $\{L_1(z), \ldots L_r(z)\}$  in one variable, one does not need this full machinery.  In this restricted case, if ${G}_1, \ldots  , {G}_s$ is a  reduced Gr\"obner basis (with respect to a graded monomial order) of the left ideal $<L_1, \ldots , L_r>$ in $k[\dd_1, \ldots ,\dd_m]$, then ${G}_1(z), \ldots  , {G}_s(z)$ is a characteristic set of $I$\footnote{In \cite{insapauer,pauer07}, the authors define and prove the existence of Gr\"obner bases for a wide class of rings. For left ideals  $J$ in $k[\dd_1, \ldots ,\dd_m]$ their definition reduces to a set $\{G_1, \ldots , G_m\} \subset J$ such that the leading term of each element of $J$ be divisible by the leading term of some $G_i$. This guarantees that  $\{G_1(z), \ldots , G_m(z)\}$ is coherent. If $\{G_1, \ldots, G_m\}$ is reduced in the usual sense of Gr\"obner bases, then $\{G_1(z), \ldots , G_m(z)\}$ is autoreduced. Using the criterion of (\cite{DAAG}, Lemma 2, p.167 or \cite{sit2000}, Ex. 8.17) and the fact that the condition of  \cite{insapauer,pauer07} guarantees that  $\{G_1(z), \ldots, G_m(z)\}$ generates  $I \cap k\{z\}_1$ where $I$ is the differential ideal generated by the $L_i(z)$ and $k\{z\}_1$ is the set of homogeneous linear differential polynomials, one can see that $\{G_1(z), \ldots , G_m(z)\}$ is a characteristic set of $I$. We thank William Sit for these references and explanations.}.
  The differential dimension polynomial can be computed from the leading terms of this set (with respect to the same monomial order, which is orderly; see Lemma 16, Ch.~0.17 and Theorem 6, Ch.~II.12 of \cite{DAAG}).  For $\Delta = \{\dd_x, \dd_y\}$ and $k$ a differentially closed $\Delta$-field, this can be made explicit. Identifying the monomial $\dd_x^i\dd _y^j$, we use the order $(i,j) > (i',j')$ if $i+j > i'+j'$ or $i+j = i'+j'$ and $i>i'$. If $ E= \{(i_1,j_1), \ldots, (i_t,j_t)\}$ represent the leading terms of a reduced Groebner basis of a left ideal of $k[\delta_x,\delta_y]$, we can assume that $i_1< i_2< \ldots <i_t$ and $j_1>j_2 > \ldots >j_t$. Let \[W = \{(i,j) \in \NX^2\ | \  i \geq i_1, j \geq j_t, (i,j) \notin   \cup_{(i',j') \in E}
 ((i',j') + \NX^2) \}\] and let $d = i_1+j_t$.  Geometrically, $W$ is the set of points in the quarter-plane $(i_1,j_t) + \NX^2$ below the ``stairs'' $\cup_{(i',j') \in E}
 ((i',j') + \NX^2)$.  Lemma 16, Ch.~0.17 and Theorem 6, Ch.~II.12 of \cite{DAAG}, imply that 
 \[ \omega(s) = \binom{s+2}{2} - \binom{s-d+2}{2} +|W| = ds +\frac{3d-d^2}{2} + |W|.\] 
\noindent For example, if the leading monomial of an operator $R$ is $\dd_x^i\dd _y^j$ with $i+j = d$, then the differential dimension polynomial of the differential ideal $[R(z)]$ generated by $R(z)$ is $\omega(s) = \binom{s+2}{2} - \binom{s-d+2}{2} = ds +\frac{3d-d^2}{2}$. Therefore the gauge of $G_R$ is $(1,d)$.  If $H \subset G_R$ is a  proper $\Delta$-subgroup, then the defining differential ideal $J$ contains a linear differential polynomial such that the leading term of its associated operator is not divisible by the leading term of $R$.  The stairs corresponding to leading terms of a characteristic set of the ideal $J$  strictly contain $(1,d)+\ZX^2$. One sees from this that the associated differential dimension polynomial either has degree $0$ or leading coefficient less than $d$. This implies that either the type of $H$ is smaller than  the type of $G_R$
 or the typical transcendence degree of $H$ is smaller than $d$. In particular, this implies that $G_R$ is strongly connected.\\[0.1in]
 The factorization in (\ref{factor1}) yields a sequence of subgroups of $G_L$ as follows.  Let $G_1$ be the solutions of $(\dd_x+1)(\dd_x +x\dd_y)(z) = 0$ and let $G_2$ be the solutions of $(\dd_x +x\dd_y)(z) = 0$. As noted above, each of these groups is strongly connected.  The quotients $G/G_1$  and $G_1/G_2$ are both isomorphic to the group defined by $(\dd_x+1)(z) = 0$. Therefore each quotient in the sequences  $G \supset G_1 \supset G_2 \supset \{0\}$ is strongly connected and has gauge $(1,1)$. { Let $H$ be one of these quotients and $K$ a proper subgroup of $H$.  If the gauge of $H$ is (1,1), then $\tau(K/H) = 0$, contradicting the fact that $H$ is strongly connected. Therefore $a_1(K) = 0$ and so $\tau(K) = 0$. }  Therefore these quotients are almost simple and so   the sequence satisfies the conclusion of 
 .\\[0.1in]
 The factorization in (\ref{factor2}) yields a sequence of groups as follows. Let $H_2$ be the group defined by $(\dd_x+1)(z) = 0$.  We then have the sequence $G \supset H_2\supset \{0\}$. The Theorem implies that the group $G/H_2$ cannot be almost simple.  In fact the proof of the theorem tells us how to refine this sequence.  Let $H_1 = H_2+G_2$ (note that this sum is  direct since the groups have a trivial intersection).  { Lemma~\ref{stronglem} implies this group is strongly connected.} The ideal defining this group is the intersection of the ideals defining each of the summands, i.e. 
\[ \langle (\dd_x+1)(z)\rangle \cap \langle (\dd_x +x\dd_y)(z) \rangle =  \langle L_1(z), L_2(z)\rangle\]
 where 
 \begin{eqnarray*}
L_1 & = &x{{\dd_x}}^{2}{\dd_y}+{x}^{2}{\dd_x}\,{{\dd_y}}^{2}-{{\dd_x}}^{2}
-{\dd_x}\,{\dd_y}+{x}^{2}{{\dd_y}}^{2}-{\dd_x}-{\dd_y}-x{\dd_y}\\
L_2 & = & {{\dd_x}}^{3}-{x}^{2}{\dd_x}\,{{\dd_y}}^{2}+3\,{{\dd_x}}^{2}+2\,x{
\dd_x}\,{\dd_y}+3\,{\dd_x}\,{\dd_y}-{x}^{2}{{\dd_y}}^{2}+2\,{\dd_x}+2\,x{\dd_y}+3\,{\dd_y}
\end{eqnarray*}
 form a Gr\"obner basis (with respect to an orderly ranking with $\dd_x > \dd_y$) of this left ideal (cf., Example 5 of \cite{grig_schwarz_04}. The above polynomials were calculated using the {\tt Ore\_algebra} and {\tt Groebner} packages in Maple and  differ slightly from those presented in \cite{grig_schwarz_04}). { One can use this representation to see directly that  $H_1$ is strongly connected.}  Looking at leading terms and using Lemma 16 of Ch.~0.17 of \cite{DAAG}, we see that the gauge is $(1,2)$. Furthermore, we see that if this ideal is contained in a larger ideal $J$ of the same gauge, then $J$ must be generated by a single operator of order 2. Using the techniques of \cite{grig_schwarz_04} one can show that  $L_1$ and $L_2$ do not have a common right factor of order 2 so  this is impossible.  Therefore $H_1$ is strongly connected of gauge $(1,2)$.  This yields the sequence $G \supset H_1 \supset H_2 \supset \{0\}$ where successive quotients are almost simple. We will now compare successive quotients of the two sequences: \\[0.1in]
 $\underline{H_2:}$ We will construct an isogeny between $G/G_1$ and $H_2$.  Consider the operator $T = (\dd_x+1)(\dd_x+x\dd_y)$.  This operator maps $G$ onto $G/G_1$, which is naturally isomorphic to the solution space of $(\dd_x+1)(z) = 0$, an almost simple (and therefore strongly connected) group. The group $T(H_2)$ must therefore be either all of $G/G_1$ or of rank less than $1$ (and therefore $\{0\}$ since $H_2$ is strongly connected).  If $T(H_2)=0$, then $H_2$ would be annihilated by both  $(\dd_x+1)(\dd_x+x\dd_y)$ and $(\dd_x+1)$, but this is impossible since the ideal $\langle (\dd_x+1)(\dd_x+x\dd_y), (\dd_x+1)\rangle$ contains $\dd_y$ and so is of rank $0$.  Therefore $T(H_2) = G/G_1$ and $T$ defines an isogeny.  In fact, since  $H_2 = \{fe^{-x} \ | f_x = 0\}$ one can show that $T$ is an isomorphism of $H_2$ onto $G/G_1$\\[0.1in]
 $\underline{H_1/H_2:}$ This group is $(H_2+G_2)/H_2$ which is isomorphic to $G_2/(H_2\cap G_2)$. The group $H_2\cap G_2$ is precisely the common zeroes of $(\dd_x+1)(z) = 0$ and $(\dd_x +x\dd_y)(z) = 0$, that is, just the element $0$.  Therefore   $G_2$ is isogenous (in fact, isomorphic) to $H_1/H_2$ .\\[0.1in]
 $\underline{G/H_1:}$ We will show that $G/H_1$ is isogenous to $G_1/G_2$. 
 To see this, consider the map $G_1 \hookrightarrow G \rightarrow G/H_1$. The kernel of this map is $G_1 \cap H_1$, that is those solutions of $(\dd_x+1)(\dd_x +x\dd_y)(z) =0$ that are the sum of a solution of  $(\dd_x+1)(z) = 0$ and  a solution of $(\dd_x +x\dd_y)(z) = 0$. This clearly contains $G_2$ but is not all of $G_1$, so we have a homomorphism $G_1/G_2 \rightarrow G/H_1$, that is not the zero  map.  Since $G_1/G_2$ is almost simple, the image of this map has gauge $(1,1)$ and the kernel has rank at most $0$. Since $G/H_1$ is almost simple, the image must be all of $G/H_1$. Therefore we have an isogeny of $G_1/G_2$ onto $G/H_1$. This isogeny has a nontrivial kernel; it is the solutions of $(\dd_x+1)(z) = 0, \dd_y(z) = 0$.   Although {\it this isogeny} has a nontrivial kernel, we claim that  $G_1/G_2$ is nonetheless isomorphic to $G/H_1$.  To show this it is enough to show the following.  Let $K_1$ be the solutions of $(\dd_x+1)(z) = 0$ and $K_2$ be the solutions of $(\dd_x+1)(z) = 0, \dd_y(z) = 0$. We have that $G_1/G_2$ is isomorphic to $K_1$ and $G/H_1$ is isomorphic (by the above) to $K_1/K_2$.  The homomorphism $\dd_y:\Ga \rightarrow \Ga$ maps $K_1$ surjectively to $K_1$  and its kernel in $K_1$ is $K_2$.  Therefore $K_1/K_2$ is isomorphic to $K_1$ and our claim is proved.
 } \end{ex}
 
 \begin{ex}\label{cartanex} {\rm We shall use the groups defined in Example~\ref{cartanexple} to construct an example of a $\Delta$-group having two normal sequences with isogenous but nonisomorphic quotients.  \\[0.1in]
 Let $\Delta = \{\dd_x, \dd_t\}$ and  $k = \frakU$ and $C$ be the $\Delta$-constants of $\frakU$.  Let $K_1$ be the solutions of $(\dd_x^2 - \dd_t)(z) = 0$ and $K_2$ be the solutions of $(\dd_x-x\dd_t)(z) = 0$.  Let $G = K_1 + K_2$. {  Lemma~\ref{stronglem} implies that this group is strongly connected.}  In fact its annihilating ideal is the intersection of $<\dd_x^2 - \dd_t>$ and $<\dd_x-x\dd_t>$ in $k\{\dd_x,\dd_t\}$ which is $<L>$ where\begin{eqnarray}
L & = & x\dd_x^3 - x^2\dd_x^2\dd_t - 2\dd_x^2 -x\dd_x\dd_t + x^2 \dd_t^2 +2\dd_t \nonumber \\
 & = &( x\dd_x - x^2 \dd_t -2)(\dd_x^2 - \dd_t) \nonumber\\
 & = & (x\dd_x^2 - x\dd_t - 2 \dd_x)(\dd_x - x \dd_t).\label{cartanexeq1}
 \end{eqnarray}
{ One sees directly that }$G$ is strongly connected since a subgroup defined by the vanishing of a single operator must be strongly connected. Consider the two sequences of subgroups: $G \supset K_1 \supset \{0\}$ and $G \supset K_2 \supset \{0\}$.  We claim that each quotient in either of these sequences is almost simple and, while $K_1$ is isogenous to $G/K_2$, it is not isomorphic to either $G/K_2$ or $K_2$.\\[0.1in]
 The almost simplicity of $K_1$ and $K_2$ follows from the observation in Example~\ref{evolex}.  In particular any quotient of $K_1$ or $K_2$ by a $\Delta$-subgroup  of smaller type will also be almost simple.  Therefore $G/K_1 \simeq K_2/(K_1 \cap K_2)$ and  $G/K_2 \simeq K_1/(K_1 \cap K_2)$ are almost simple. Clearly $K_1$ is isogenous to $G/K_2$ and $K_2$ is isogenous to $G/K_1$.\\[0.1in]
We now claim that $K_1$  is not isomorphic to either $G/K_2$ or $K_2$.  Note that $K_2$ has typical differential transcendence degree 2 while $K_2$ has  typical differential transcendence degree 1, so these two groups cannot be isogenous.  From (\ref{cartanexeq1}), one sees that $G/K_2$ is isomorphic to the group of solutions of $(x\dd_x^2 - x\dd_t - 2 \dd_x)z=0$.  This is the group $H_x$ of Example~\ref{cartanexple} and $K_1$ is the group $H$ of the Heat Equation in this latter example.  As we have shown in Example~\ref{cartanexple}  these groups are not isomorphic.}\end{ex}
\section{Almost Simple $\Delta$-Groups}\label{section_asg} In this section we will develop some facts concerning almost simple $\Delta$-groups.   We will use the terms $\Delta$-group, $\Delta$-subgroup, etc. to mean that the objects are all defined over $\frakU$.  Similarly almost simple, irreducible, etc.~refer to these properties with respect to $\frakU$, a differentially closed $\Delta$-field.  
\subsection{Quasisimple Linear Algebraic Groups are Almost Simple $\Delta$-groups}  We begin with the following lemma. 
\begin{lem}\label{lemnorm} Let $(\frakU, \Delta)$ be a differentially closed differential field, $\Delta' \subset \Delta$ and $G$ a non-commutative  connected algebraic subgroup of $\GL_n(\frakU)$ with finite center, that is defined over the field $C'$ of $\Delta'$-constants of $\frakU$. Then the $\Delta'$-subgroup $G(C')$ of $G$ equals its own normalizer in $G$.
\end{lem}
\begin{proof} First note that since $C'$ is algebraically closed, $G(C')$ is Zariski dense in $G$ and is non-commutative and connected.  Let $N$ be the normalizer of $G(C')$ in $G$. For $g \in N$ and $c \in G(C')$, we have that $g^{-1} c g = c_1 \in G(C')$. For any $\delta' \in \Delta'$, differentiating this latter equation yields $-g^{-1}\delta'g\cdot g^{-1} cg + g^{-1}c\delta'g = 0$ and therefore 
\[ c^{-1}\ell\delta'g \ c = \ell\delta'g\]
where $\ell\delta'g = \delta'g \cdot g^{-1}$ is called the logarithmic derivative of $g$.  It is well known that $\ell\delta'g \in\frakg,$ the lie algebra of $G$ (see \cite{DAAG}, Ch.~V.22 or \cite{MiSi2002}, Sec.5, for an exposition of properties of the logarithmic derivative in the linear case). Since $G(C')$ is Zariski dense in $G$, this latter equation implies that every element of the $\frakU$-span $W$ of $\ell\delta'g$ is  invariant under the adjoint action of $G$ on $\frakg$.  Corollary 3.2 of \cite{hochschild} implies that $W$ is annihilated by the adjoint action of $\frakg$ on $\frakg$, that is $W$ lies in the center of $\frakg$.  Theorems 3.2 and 4.2 in Chapter 4  of \cite{hochschild} imply that the center of $\frakg$ is the lie algebra of the center of $G$.  Since the center of $G$ is finite, we must have that $\ell\delta'g = 0$, that is $\delta'g  = 0$ and so $g \in G(C')$.\end{proof}
\begin{prop} Let $(\frakU, \Delta)$ be a differentially closed differential field with $\Delta$-constants $C$ and let $G \subset \GL_n(\frakU)$ be a non-commutative quasisimple linear algebraic group defined over $ C$. If $N$ is a proper normal $\Delta$-subgroup of $G$, then $N$ is finite.  In particular, if $G$ is quasisimple (resp., simple) as an algebraic group, then it is quasisimple (resp., simple) as a $\Delta$-group and therefore in either case is almost simple.
\end{prop}
\begin{proof} Let $N$ be a proper normal $\Delta$-subgroup of $G$ and $N^0$ its identity component (in the Kolchin topology).  $N^0$ is again normal in $G$. Let $H \subset G$ be the Zarski closure of $N^0$. One sees that $H$ is connected in the Zariski topology and is again normal in $G$. We wish to show that $H$ is finite. If not, then, since $G$ is quasisimple as an algebraic group, we must have that $H = G$.  This implies that $N^0$ is a  Zariski dense $\Delta$-subgroup of $G$, connected in the Kolchin topology.  Theorem 19  of \cite{cassidy6} states that there exists a finite set of derivations $\Delta'$ that are $\frakU$-linear combinations of the elements of $\Delta$ such that $N^0$ is conjugate in $\GL_n(\frakU)$ to $G(C')$ where $C' = \{c \in \frakU \ | \ \delta'(c) = 0, \ \forall \delta'\in \Delta'\}$.  Expanding $\Delta$ if necessary we may assume $\Delta' \subset \Delta$ (this does not change the hypotheses or conclusions). Lemma~\ref{lemnorm} implies that $N^0$ is its own normalizer so $N^0 = G$, contradicting the fact that $N$ is a proper subgroup of $G$.\end{proof}
\subsection{Almost Simple Linear $\Delta$-Groups} We now restrict ourselves to {\it linear} differential algebraic groups.  In this situation , we are able to derive additional properties of almost simple groups and, in the case of ordinary differential fields, give a strong classification of almost simple linear $\Delta$-groups. 
\begin{prop}\label{comsimpleprop} Let $G$ be a  nontrivial commutative simple  linear $\Delta$-group.  Then either $G$  is finite of prime order or is isomorphic to $\Ga(C)$.
\end{prop}  
\begin{proof}Assume $G \subset \GL_n(\frakU)$ for some positive integer $n$. If $G$ is finite then it clearly must be of prime order.  Let us now assume that $G$ is infinite.  Since the identity component $G^0$ is a normal $\Delta$-subgroup of finite index, we must have that  $G= G^0$ so $G$ is connected.  Therefore $\Gbar$, the Zariski closure of $G$ in $\GL_n(\frakU)$, is also connected and commutative. Since $\frakU$ is algebraically closed, we may assume that $\Gbar = \Gm^p(\frakU)\times \Ga^q(\frakU)$ for some non-negative integers $p,q$.   Assume that $p>0$ and let $\pi$ be the projection of $\Gbar$ onto some factor of $\Gm^p(\frakU)$.  We have that $\pi(G(\frakU))$ is Zariski dense in $\Gm(\frakU)$ and so by Proposition 31 of \cite{cassidy1}, $\pi(G(\frakU))$ contains $\Gm(C)$. This would imply that $G$ would have a proper nontrivial normal $\Delta$-subgroup.  Therefore $p = 0$ and $\Gbar = \Ga^q(\frakU)$ for some $p$.  Let $0 \neq a = (a_1, \ldots , a_q) \in G$. The subgroup  $Ca$ is $\Delta$-closed and $\Delta$-isomorphic to   $\Ga(C)$. Therefore $G$ must be isomorphic to $\Ga(C)$.\end{proof}
If we assume that $k$ is algebraically closed and $G$ is a commutative $k$-simple $\Delta$-$k$-group, then, we can only say that $G$ is either finite or isomorphic to a proper $\Delta$-$k$-subgroup of $\Ga(\UX)$.  For, if $G$ is Zariski dense in $\Ga^q (\UX)$, let $\pi$ be the projection of $\Ga^q (\UX)$ onto 
$\Ga(\UX)$.  Then, $G$ is delta-isomorphic to $\pi(G)$, but, we do not know that the latter group contains a point rational over $k$.\\[0.1in]
In the theory of linear algebraic groups, one knows that the commutator subgroup of a linear algebraic group is closed in the Zariski topology and therefore  that a quasisimple linear algebraic group is perfect, that is, it equals its commutator subgroup (in fact, every element of a semisimple algebraic group is a commutator \cite{Re64}). As we have already mentioned following Definition~\ref{defcom}, in contrast to the algebraic case, the commutator subgroup of a differential algebraic group need not be closed in the Kolchin topology (Ch.~IV.5, \cite{kolchin_groups}). Nonetheless, we have the following proposition
\begin{prop} \label{d(G)}Let $G$ be an infinite almost simple linear $\Delta$-group. Then either $G$ is commutative or $D_\Delta(G) = G$.  \end{prop}
\begin{proof} Assume $G$ is not commutative and suppose that $D_\Delta(G) \neq G$. Since $D_\Delta(G)$ is a normal $\Delta$-subgroup of $G$, Corollary~\ref{simplecor} implies that $D_\Delta(G)$ lies in the center $Z(G)$ of $G$. Therefore $G/Z(G)$ is commutative. Since every proper normal $\Delta$-subgroup of $G$ is central, the group $G/Z(G)$ is simple.  Proposition~\ref{comsimpleprop} implies that $G/Z(G)$ is isomorphic to $\Ga(C)$. In particular, $0 = \tau(G/Z(G)) = \tau(G)$.  This implies that $\tau(Z(G)) = -1$, that is, $Z(G)$ is finite. Since $D_\Delta(G)$ is connected we must have that $D_\Delta(G)$ is trivial, so $G$ is commutative, a contradiction.
\end{proof}
Suppose G is an infinite almost simple linear $\Delta$-group.  An argument based on the proof of Proposition~\ref{d(G)} breaks down at the last step.  We do not know a priori that a simple $\Delta$-subgroup of $\Ga(\UX)$ has type $0$ unless $\UX$ is an ordinary differential field. We also do not have an example of a non-commutative almost simple linear differential algebraic group whose commutator subgroup is not closed in the Kolchin topology. In fact, as we show below in Proposition~\ref{ordcomm}, non-commutative almost simple linear $\Delta$-groups of {\it type at most $1$} are perfect.\\[0.2in]
We derive one more property of the center. Let $G \subset \GL_n$ be a $\Delta$-group  and $\Gbar$  its Zariski closure. Let $\Gbar_u$ be the unipotent radical and $\Gbar_r$ be the solvable radical of $\Gbar$.  One can easily show that $\Gbar_u \cap G$ is the unique maximal normal unipotent differential algebraic subgroup of $G$ and $(\Gbar_r \cap G)^0$ is the unique maximal connected solvable $\Delta$-subgroup of $G$. We shall denote $\Gbar_u \cap G$  by $G_u$  and $(\Gbar_r \cap G)^0$ by $G_r$ and refer to these as the unipotent radical and  radical respectively.  If $G$ is defined over $k$, so are $G_u$ and $G_r$.
\begin{prop} \label{center}Let $G$ be an infinite almost simple   linear $\Delta$-group.  
\begin{enumerate}
\item If $G$ is non-commutative, then  the unipotent radical $G_u$ of $G$ equals the identity component $(Z(G))^0 $ of the center of $G$.
\item If $G$ is commutative then either $G$ is isomorphic to $G_m(C)$ or $G$ is isogenous to an almost simple subgroup of $G_a(\frakU)$.\end{enumerate}
\end{prop}
\begin{proof} Let $G \subset \GL_n$ and let $\Gbar$ be the Zariski closure of $G$. \\[0.1in]
$1.$~Assume that $G$ is a non-commutative group.  The commutator subgroup $(\Gbar,\Gbar)$ of $\Gbar$ is a Zariski closed subgroup of $\Gbar$ and, {\it a fortiori}, is Kolchin closed.  Therefore it contains $D_\Delta(G) = G$. Since $(\Gbar,\Gbar)$ is Zariski closed, it must also contain $\Gbar$. Therefore $\Gbar = (\Gbar,\Gbar)$ is perfect.\\[0.1in]
The group $\Gbar$ has a Levi decomposition
\[\Gbar = \Gbar_u \rtimes P\]
where $P$ is a reductive algebraic group.  Since $\Gbar$ is perfect, $P$ is perfect. The solvable radical $P_r$ of $P$ has finite intersection  with $(P,P)$ (Lemma, p.~125 \cite{humphreys}) so $P_r$ is finite.  Therefore $P$ must be semisimple. This implies that the radical $\Gbar_r$ of $\Gbar$  equals $\Gbar_u$. Therefore, $G_r = G_u$.  This implies that $Z(G)^0 \subset G_u$.  Since $D_\Delta(G) = G$, we have that $G_u$ is a proper normal $\Delta$-subgroup of $G$ and so is central.  Therefore $Z(G)^0 = G_u$.\\[0.1in]
$2.$~Assume $G$ is commutative so $\Gbar $ is also commutative.  Therefore $\Gbar \simeq \Gm(\frakU))^s \times (\Ga(\frakU))^t $ for some integers $s,t$. The map $l\Delta:\Gm(\frakU) \rightarrow (\Ga(\frakU))^m$ given by $l\Delta(u) = (\dd_1(u)u^{-1}, \ldots, \dd_m(u)u^{-1})$ is a $\Delta$-homomorphism so we have a $\Delta$-homomorphism 

\[\phi = (l\Delta^s, id^t):G \rightarrow (\Ga(\frakU)^{ms} \times (\Ga(\frakU))^t .\] 

\noindent The kernel of this map is $G \cap(\Gm(C))^s \subset (\Gm(\frakU))^s \subset (\Gm(\frakU))^s \times (\Ga(\frakU))^t$.  Let $\pi: (\Ga(\frakU)^{ms} \times (\Ga(\frakU))^t \rightarrow \Ga(\frakU)$ be a projection onto one of the $\Ga(\frakU)$ factors.  Since $G$ is almost simple, the group $\pi\circ\phi (G)  \subset  \Ga(\frakU)$ is a nontrivial almost simple $\Delta$-group  or is the trivial group.  If it is a nontrivial almost simple  group, then    the map $\pi\circ\phi$ is an isogeny since its kernel in $G$ is a proper subgroup.  If this group is trivial for all projections $\pi$, then $G$ must be a subgroup of $\Gm(C)^s $. This implies that $G$ is a connected algebraic subgroup of $\Gm(C)^s$. Such groups  are isomorphic to $(\Gm(C)^\ell$ for some $\ell$.  Since $G$ is almost simple, we have that $\ell = 1$.  \end{proof}
\begin{remarks} {\rm 1.~We cannot strengthen the conclusion of part $2.$ above. Let $(\UX,\dd)$ be an ordinary differential field and let $G \subset \Gm(\frakU) \times \Ga(\frakU)$ be the graph of the homomorphism $\l\dd : \Gm(\frakU) \rightarrow \Ga(\frakU)$. $G$ contains the group $\Gm(C) \times 0$ and so cannot be unipotent.  Therefore $G$ is isogenous to $\Ga(\frakU)$ (via the projection onto the second factor) but is not isomorphic to any subgroup of $\Ga(\frakU)$.\\
$2.$~For information concerning commutative unipotent differential algebraic groups, see \cite{cassidy3}. Nonetheless, we have no general classification of almost simple proper $\Delta$-subgroups of $\Ga(\frakU)$ (except in the ordinary differential case; see below).} \end{remarks}
The above propositions yield the following theorem. We shall denote by $\frakU\cdot\Delta$ the $\frakU$ span of $\Delta$. The elements of $\frakU\cdot\Delta$ are derivations on $\frakU$. 
\begin{thm} \label{almostsimplethm} Let $G$ be an almost simple linear $\Delta$-group of type $\tau$. \begin{enumerate}
\item If $G$ is non-commutative, then: 
\begin{enumerate}
\item $G = D_\Delta(G)$,
\item $G_u = (Z(G))^0$, and
\item There exists a finite commuting subset $\Delta' \subset \frakU\cdot \Delta$, $|\Delta'| = m-\tau$, such that $G/Z(G) \simeq H(C_{\Delta'})$ where $H$ is a simple algebraic group defined over $\QX$ and $C_{\Delta'} \subset \frakU$ is the field of $\Delta'$-constants.
\end{enumerate}
\item If $G$ is commutative, then $G$ is either isogenous to a $\Delta$-subgroup of $\Ga(\frakU)$ or is isomorphic to $\Gm(C)$. 
\end{enumerate}
\end{thm}
\begin{proof} Statement 1(a) follows from Proposition~\ref{d(G)}.  Statement 1(b) follows from Proposition~\ref{center}.1. Statement 2.~follows from Proposition~\ref{center}.2. To verify statement 1(c), note that Corollary~\ref{simplecor2} states that $G/Z(G)$ is a simple $\Delta$-group.  Its Zariski closure must be a simple linear algebraic group. The main result of $\cite{cassidy6}$ gives the conclusion of 1(c). \end{proof}

\subsection{Almost Simple Linear $\Delta$-Groups of Differential Type $\leq 1$}
We note that this includes {\it ordinary} linear $\Delta$-groups.  Restricting ourselves to $\Delta$-groups of differential type at most $1$ allows us to  sharpen Proposition~\ref{d(G)}.
\begin{prop}\label{ordcomm} Let $G$ be an infinite almost simple linear $\Delta$-group with $\tau(G) \leq 1$.  Then either $G$ is commutative or  perfect.\end{prop}
\begin{proof} Assume that $G$ is not commutative. Corollary~\ref{simplecor2} states that $G/Z(G)$ is a simple linear differential algebraic group.  Theorem 17 of \cite{cassidy6} implies that $G/Z(G)$ is differentially isomorphic to a group $G'(F)$ where $G'$ is a simple non-commutative algebraic group and $F$ is a field of constants with respect to a set of $\frakU$-linear combinations of elements of $\Delta$. For an arbitrary group $H$, let $C_m(H) = \{ x_1y_1x_1^{-1}y_1^{-1} \cdots x_my_mx_m^{-1}y_m^{-1} \ | \ x_i, y_i \in H\}$.  Since $G'$ is simple, we know that $G' = C_1(G')$ (cf., \cite{Re64}). We will show that $G = C_s(G)$ for some $s$. \\[0.1in]
Let $X_m$ be the Kolchin closure of $C_m(G)$ in $G$. Let $\pi:G \rightarrow G/Z(G)$ be the canonical projection. Since $G' = C_1(G')$, we have that $\pi(C_1(G)) = \pi(G)$. Therefore $\pi(X_1) = G/Z(G)$. Since the  group $Z(G)$ has smaller type than $G$, we have that $\tau(G) = \tau(G/Z(G))$ and $a_\tau(G) = a_\tau(G/Z(G))$. Therefore $\tau(X_1) = \tau(G)$ and furthermore $a_\tau(X_1) = a_\tau(G)$.\\[0.1in]
  For each $i = 1, 2, \ldots $ let $\omega_i(s)= a_i \binom{s+1}{1} +b_i$ be the differential dimension polynomial of a generic point of $X_i$ and $\omega(s) = a\binom{s+1}{1} +b$ be the differential dimension polynomial of $G$ (we include the possibility that $a = 0$).  Since $X_1 \subset X_2 \subset  \ldots \subset G$, we have that $\omega_1(s) \leq \omega_2(s) \leq \ldots \leq \omega(s)$. As we have just shown, $a_1 = a$  so we must have $a_1 = a_2 = \ldots = a$.  Therefore, $b_1 \leq b_2 \leq \ldots \leq b$.  Since these are all integers, we must have that  $b_r= b_{r+1} = \ldots $ for some $r$.  This implies that $\omega_r(s) = \omega_{r+1}(s) = \dots$ and so $X_r = X_{r+1} = \ldots . $ (cf., Proposition 2, Ch.III.5 \cite{DAAG}).\\[0.1in]
We will now show (following a similar proof for algebraic sets) that $G = C_{2r}$. For any $ c \in C_r(G)$ the map $d\mapsto dc$ sends $C_r(G)$ to $C_{2r}(G)$ and therefore sends $X_r$ to $X_r$.  A similar argument shows that for any $c \in X_r$, we have that multiplication by $c$ on the left sends $X_r$ to $X_r$.  Therefore $X_r$ is a $\Delta$-subgroup of $G$.  Since $C_r(G)$ is invariant under conjugation by elements of $G$, we have that $X_r$ is a normal subgroup of $G$ of the same type and so must equal $G$. Since $C_r$ is constructible, it contains an open subset of its closure, that is an open subset of $G$.  For any $g \in G$, we have that $gC_r^{-1}$ intersects $C_r$.  Therefore $g \in C_r\cdot C_r = C_{2r}$. \end{proof}

\noindent We can also  strengthen Theorem~\ref{almostsimplethm}.

\begin{thm} \label{almostsimplethm2} Let  $G$ be an infinite almost simple linear $\Delta$-group with $\tau(G) \leq 1$. Let $|\Delta | = m$. \begin{enumerate}
\item If $G$ is non-commutative, then there exists a quasisimple algebraic group $H$ defined over $\QX$ and a commuting linearly independent  set of $m-1$ derivations $\Delta' \subset \frakU\Delta $  such that $G$ is $\Delta$-{isomorphic} to either  $H(C)$ or $H(C')$ where $C$ are the $\Delta$-constants  and $C'$ are the $\Delta'$ constants of $\frakU$.
\item If $G$ is commutative, then $G$ is either isogenous to an almost simple $\Delta$-subgroup $G' \subset \Ga(\frakU)$  with $\tau(G') = 1$ or is $\Delta$-isomorphic to $\Ga(C)$ or  $\Gm(C)$. 
\end{enumerate}
\end{thm}
\begin{proof} 1.~Let $G$ be a non-commutative almost simple $\Delta$-group with $\tau(G) \leq 1$. We shall first show that $Z(G)$ is finite. \\[0.1in]
 If $\tau(G) = 0$, then $\tau(Z) = -1$ and so $Z(G)$ is finite. Now assume $\tau(G) = 1$. We shall rely heavily on the results of \cite{alch99}.  Theorem~\ref{almostsimplethm}.1(c) implies that $G/Z(G)\simeq H(F)$, where $H$ is a simple algebraic group defined over $\QX$ and $F$ is an ordinary differentially closed field with respect to some derivation $\delta$ that is a $\frakU$-linear combination of elements of $\Delta$.  The field of constants  of $F$ is the field $C$ of $\Delta$-constants of $\frakU$. Therefore there is an exact sequence 
 \[ 1 \longrightarrow Z(G) \stackrel{incl.}{\longrightarrow} G \stackrel{\alpha}\longrightarrow H \longrightarrow 1,\] 
 where $\alpha$ is a $\Delta$-homomorphism.
Let $\tilde{G} = \alpha^{-1}(H(C))$.  We then have the exact sequence

 \[ 1 \longrightarrow Z(G) \stackrel{incl.}{\longrightarrow}\tilde{G} \stackrel{\alpha}\longrightarrow H(C) \longrightarrow 1.\] 
Since the type of $Z(G)$ is at most $0$ and the type of $H(C)$ is $0$, we must have that the type of $\tilde{G}$ is $0$, that is, $\tilde{G}$ is a group of finite Morley rank. The group $H(C)$ is a simple algebraic group of constant matrices and so is a simple $\dd$-group as well. Therefore every normal subgroup of $\tilde{G}$ is in $Z(G)$. Proposition~\ref{ordcomm} implies that $\tilde{G}$ is perfect.  Therefore the results of  \cite{alch99} imply that $Z(G)$ is finite. \\[0.1in]
Since $Z(G)$ is finite and $G/Z(G)$ is a simple group, we have that $G$ is a quasisimple $\Delta$-group. Theorem~\ref{appendixthm} in the Appendix states that  there exists a quasisimple algebraic group $H$ defined over $\QX$ and a commuting linearly independent  set of  derivations $\Delta' \subset \frakU\Delta $  such that $G$ is $\Delta$-{isomorphic}  $H(C')$ where $C'$ are the $\Delta'$ constants of $\frakU$. Since the type of $G$ is at most $1$, we have that either $\Delta'$ has $m$ or $m-1$ elements.  If $\Delta'$ has $m$ elements, then the $\Delta'$ constants and the $\Delta$-constants coincide so $G$ is $\Delta$-isomorphic to $H(C)$, where $C$ is the field of $\Delta$-constants.\\[0.1in]
2.~From Theorem~\ref{almostsimplethm}, we know that a commutative almost simple $\Delta$-group $G$ is either isogenous to a  subgroup $G'$ of $\Ga(\frakU)$ or isomorphic to $\Gm(C)$.  If $\tau(G) = 0$, let $u$ be a nonzero element of $G'$. The group $H = C\cdot u$ is a $\Delta$-subgroup of $G'$  with $\tau(G') = 0$. Since $G'$ is also almost simple, we must have $G' = H$ and so $G'$ is isomorphic to $\Ga(C)$. Therefore $G$ is isogenous to $\Ga(C)$.  \\[0.1in]
We will now show that any almost simple group $G$ that is isogenous to $\Ga(C)$ is isomorphic to $\Ga(C)$.  Since $G$ is isogenous to $\Ga(C)$, there exists a strongly connected $\Delta$-group $G'$ and surjective $\Delta$-homomorphisms $\alpha_1:G'\rightarrow G$, $\alpha_2:G'\rightarrow \Ga(C)$ where $\alpha_1$ and $\alpha_2$ have finite kernels. If we can show that $\alpha_2$ has trivial kernel, then $G'$ would be isomorphic to 
$\Ga(C)$ and so $\alpha_1$ would have trivial kernel as well. This would further imply that $\alpha_1$ is an isomorphism.  \\[0.1in]
Therefore we must show the following: If there is a surjective $\Delta$-homomorphism $\alpha:G \rightarrow \Ga(C)$ with finite kernel $H$ then $H$ is trivial.   Since $G$ is also isogenous to $\Ga(C)$, $G$ is commutative (Corollary \ref{corcom}).\\[0.1in]
Let $|H| = n$.  Since $\Ga$ is torsion free, we have that the torsion subgroup of $G$ is $H$.  Therefore, $H$ is the kernel of the homomorphism $\gamma:G\rightarrow G, \ \gamma(g) = g^n$. The differential type of $\gamma(G)$ must therefore be the same as the type of $G$. Since $G$ is strongly connected, we have $\gamma(G)  = G$.  This implies that if  $h \in H, \ h \neq 1$, there exists a $g \in G \backslash H$ such that $g^n = h$. This element $g$ would then be a torsion element not in $H$, a contradiction.  Therefore $H$ is trivial.\end{proof}
\begin{cor}\label{almostsimplecor} Assume that $|\Delta| = 1$, that is, $\frakU$ is an ordinary differential field.  Let $G$ be an almost simple $\Delta$-group. \begin{enumerate}
\item If $G$ is non-commutative, then there exists a quasisimple algebraic group $H$ defined over $\QX$ such that  $G$ is $\Delta$-isomorphic to $H(\frakU)$ or $H(C)$. 
\item If $G$ is commutative, then either $G$ is isogenous to $\Ga(\frakU)$ or $\Delta$-isomorphic to $\Ga(C)$ or $\Gm(C)$. \end{enumerate} \end{cor}
\begin{proof} Corollary~\ref{almostsimplecor}.1 follows immediately from Theorem~\ref{almostsimplethm2}.1.  Corollary~\ref{almostsimplecor}.2 follows from Theorem~\ref{almostsimplethm2}.2 once one notes that any proper $\Delta$-subgroup of $\Ga(\frakU)$ has type $0$. Note that 2 is no longer true for almost $k$-simple $\Delta$-groups.\end{proof}
Example~\ref{Gaex} shows that even in the case of ordinary differential fields, there are many nonisomorphic almost simple $\Delta$-groups that are isogenous to $\Ga(\frakU)$.  
\section*{Appendix: Quasisimple Linear $\Delta$-Groups} In this appendix we show how the results of \cite{cassidy6}\footnote{See also \cite{cassidy_quasi} for a useful discussion of the results of \cite{cassidy6}.} give the following result that was needed in the proof of Theorem~\ref{almostsimplethm2}. 

\begin{thm}\label{appendixthm} Let $G$ be a quasisimple linear $\Delta$-group. There exists a quasisimple algebraic group $H$ defined over $\QX$,  a commuting basis $\tilde{\Delta} $ of $\frakU\Delta$  such that $G$ is $\Delta$-isomorphic to $H(C')$ where $C'$ are the $
\Delta'$ constants of $\frakU$ for some $\Delta' \subset \tilde\Delta$.
\end{thm}

This result will follow from the following three propositions.  The first two appear in \cite{cassidy6}.  Note that in this latter paper, the author uses the term {\it simple} to mean quasisimple (cf, \cite{cassidy6}, p.~222). We use the same notation as in Theorem~\ref{appendixthm}
\begin{prop}[\cite{cassidy6}, Corollary 1, p.~228]\label{appendixprop1} Let $G$ be a connected Zariski dense $\Delta$-subgroup of a semisimple algebraic subgroup $H$ of $\GL_n(\frakU)$. Then, $G$ is quasisimple if and only if $H$ is quasisimple.
\end{prop}
\begin{prop}[\cite{cassidy6}, Theorem 19, p.~232]\label{appendixprop2} Let $G$ be a connected Zariski dense $\Delta$-subgroup of a quasisimple algebraic group $H$ where $H$ is defined over the constants of $\Delta$. Then, $G$ is conjugate to $H(C')$ where $C'$ are the $\Delta'$ constants of some commuting linearly independent set of derivations of $\frakU\Delta$.
\end{prop} 
\begin{prop}\label{appendixprop3} Let $G$ be a quasisimple $\Delta$-subgroup of $\GL_n(\frakU)$.  Then, there exists a $\Delta$-rational isomorphism $\phi$ from $G$ onto a Zariski dense $\Delta$-subgroup of a quasisimple algebraic group $H$.
\end{prop}
\begin{proof} From Theorem 4.3, Chapter VIII of \cite{humphreys} we have that the Zariski closure $\overline{G}$ of $G$ can be written as a semidirect product
\[\overline{G}_u\rtimes H,\]
where $\overline{G}_u$ is the unipotent radical of $\overline{G}$ and $H$ is a reductive algebraic subgroup of $\overline{G}$ and is $H$ is connected. \\[0.1in]
We claim that $H$ is semisimple. First note that $\overline{G}$ is perfect, that is, $\Gbar = (\Gbar,\Gbar)$.  To see this,  Theorem~\ref{almostsimplethm} implies that $G = D_\Delta(G)$. The group $(\Gbar,\Gbar)$ is Zariski closed and so is Kolchin closed. Since $(\Gbar,\Gbar)\subset \Gbar$ contains $(G,G)$, it contains $D_\Delta(G) = G$ and therefore must equal $\Gbar$. Let $\phi:\Gbar \rightarrow H$ be the projection with kernel $\Gbar_u$. We have $H = \phi(\Gbar) = \phi((\Gbar,\Gbar)) = (H,H)$, that is, $H$ is perfect. The solvable radical $H_r$ has finite intersection with $(H,H)$ (Lemma, p.125 \cite{humphreys}), so $H_r$ is finite. Therefore $H$ is semisimple.\\[0.1in]
We now claim that $\phi$ restricted to $G$ is an isomorphism of $G$ onto a Zariski dense subgroup of $H$. The group $\phi(\Gbar)$ is Zariski closed, from which it follows that $H=\phi(\Gbar)$ is the Zariski closure of $\phi(G)$. To see that $\phi$ is injective on $G$, note that, since $G$ is quasisimple, $\ker \phi|_G$ is either finite or all of $G$. If  $G = \ker \phi|_G$, we would have $G\subset \Gbar_u$, and so $G$ would not be quasisimple. Therefore $\ker \phi|_G \subset \Gbar_u $ is a finite unipotent group and so must be trivial. \\[0.1in]
Since $\phi(G)$ is a connected quasisimple  Zariski dense $\Delta$-subgroup of the semisimple group $H$, Proposition~\ref{appendixprop1} implies that $H$ is quasisimple.
\end{proof}
\noindent {\bf Proof of Theorem~\ref{appendixthm}.} From Proposition~\ref{appendixprop3}, we may assume that $G$ is a connected Zariski dense $\Delta$-subgroup of a quasisimple algebraic group $H$ defined over $\frakU$.  A fundamental theorem of Chevalley (\cite{chevalley_notes}\footnote{See \cite{borelchevalley} (especially 3.3(6) and 3.3(1)) for a survey of these notes  and \cite{cassidy_quasi} for a discussion of the literature reproving and extending this result.}), states that $H$ is isomorphic to an algebraic group defined over $\QX$. Therefore, we may further assume that $H$ itself is a quasisimple group defined over $\QX \subset  C$, the constants of $\Delta$.  Proposition~\ref{appendixprop2} implies that  $G$ is $\Delta$-isomorphic to $H(C')$ where $C'$ are the $\Delta'$ constants of some commuting linearly independent set of derivations of $\frakU\Delta$. Proposition 7, Chapter 0 of \cite{kolchin_groups} states that we may extend $\Delta'$  to a basis $\tilde\Delta $ of $\frakU\Delta$. \hfill \QED
\bibliographystyle{plain}
\newcommand{\SortNoop}[1]{}\def\cprime{$'$}



\end{document}